\documentclass[draft,12pt,a4paper]{amsart}
\usepackage{amssymb}
\usepackage[T2A]{fontenc}
\usepackage[cp1251]{inputenc}
\usepackage[english]{babel}
\usepackage{srcltx}

\advance\textwidth20mm \advance\hoffset-10mm
\advance\textheight20mm \advance\voffset-10mm
\theoremstyle{plain}
\newtheorem{theorem}{Theorem}[section]
\newtheorem{corollary}[theorem]{Corollary}
\newtheorem{lemma}[theorem]{Lemma}

\theoremstyle{remark}
\newtheorem{remark}{Remark}[section]

\numberwithin{equation}{section}

\newcommand\<{\langle}
\renewcommand\>{\rangle}

\newcommand\wt{\widetilde}

\newcommand\supp{\operatorname{supp}}

\newcommand\supvrai{\operatorname*{supvrai}}
\newcommand\Cdot{{\mskip2mu{\cdot}\mskip2mu}}

\makeatletter
\newcommand{\leqnomode}{\tagsleft@true\let\veqno\@@leqno}
\newcommand{\reqnomode}{\tagsleft@false\let\veqno\@@eqno}
\makeatother

\begin{document}

\title[Riesz potential and maximal function for Dunkl transform]
{Riesz potential and maximal function\\ for Dunkl transform}

\author{D.~V.~Gorbachev}
\address{D.~Gorbachev, Tula State University,
Department of Applied Mathematics and Computer Science,
300012 Tula, Russia}
\email{dvgmail@mail.ru}

\author{V.~I.~Ivanov}
\address{V.~Ivanov, Tula State University,
Department of Applied Mathematics and Computer Science,
300012 Tula, Russia}
\email{ivaleryi@mail.ru}

\author{S.~Yu.~Tikhonov}
\address{S. Tikhonov, ICREA, Centre de Recerca Matem\`{a}tica, and UAB\\
Campus de Bellaterra, Edifici~C
08193 Bellaterra (Barcelona), Spain}
\email{stikhonov@crm.cat}

\date{\today}

\thanks{The first and the second authors were partially supported by RFBR
N\,16-01-00308. The third author was partially supported by MTM 2014-59174-P,
2014 SGR 289,  and by the CERCA
Programme of the Generalitat de Catalunya.}

\begin{abstract}
We study weighted $(L^p, L^q)$-boundedness properties of Riesz potentials 
 and fractional maximal functions for the Dunkl transform.
 In particular, we obtain the weighted Hardy--Littlewood--Sobolev type inequality and weighted week $(L^1, L^q)$ estimate.
 We find a sharp constant in the weighted $L^p$-inequality, generalizing the results of W.~Beckner  and S.~Samko.
\end{abstract}

\maketitle

\section{Introduction}

Let $\mathbb{R}^{d}$ be the real Euclidean space of $d$ dimensions
 equipped with a scalar
product $\<x,y\>$ and a norm $|x|=\sqrt{\<x,x\>}$.
 Let
 $d\mu(x)=(2\pi)^{-d/2}\,dx$ be the normalized Lebesgue measure, $L^p(\mathbb{R}^{d})$, $1\leq p<\infty$, be the Lebesgue space with the norm $\|f\|_p=\bigl(\int_{\mathbb{R}^{d}}|f|^p\,d\mu\bigr)^{1/p}$, and $\mathcal{S}(\mathbb{R}^d)$ be the Schwartz space.
 The Fourier transform  is given by
\[
\mathcal{F}(f)(y)=\int_{\mathbb{R}^{d}}f(x)e^{-i\<x,y\>}\,d\mu(x).
\]

Throughout the paper, we will assume that $A\lesssim B$ means that $A\leq C B$ with a constant $C$ depending only on nonessential parameters. For $p\ge1$,\, $p'=\frac{p}{p-1}$ is the H\"{o}lder conjugate and $\chi_E$ is the characteristic function of a set $E$.

The Riesz potential operator or fractional integral $I_{\alpha}$ is defined by
\[
I_{\alpha}f(x)=(\gamma_{\alpha})^{-1}\int_{\mathbb{R}^d}f(y)|x-y|^{\alpha-d}\,d\mu(y)=
(\gamma_{\alpha})^{-1}\int_{\mathbb{R}^d}\tau^{-y}f(x)|y|^{\alpha-d}\,d\mu(y),
\]
where $0<\alpha<d$, $\gamma_{\alpha}=2^{\alpha-d/2}\Gamma(\alpha/2)/\Gamma((d-\alpha)/2)$, and
   $\tau^yf(x)=f(x+y)$ is the translation operator.
Such operator was first investigated by O.~Frostman \cite{Fro35}. Several important properties of the potential were obtained by M.~Riesz \cite{Rie49}.

The weighted $(L^p, L^q)$-boundedness of Riesz potentials is given by the following
Stein--Weiss inequality  
 \begin{equation}\label{eq1.1}
\big\||x|^{-\gamma}I_{\alpha}f(x)\big\|_q\leq \mathbf{c}(\alpha,\beta,\gamma,p,q,d)\big\||x|^{\beta}f(x)\big\|_p
\end{equation}
with the sharp constant $\mathbf{c}(\alpha,\beta,\gamma,p,q,d)$ and $1<p\le q<\infty$.
Sufficient  conditions for the finiteness of  $\mathbf{c}(\alpha,\beta,\gamma,p,q,d)$ are well known.
\begin{theorem}\label{thm1.1}
Let \, $d\in \mathbb{N}$, $1\le p\le q<\infty$,
$\gamma<\frac{d}{q}$,  $\gamma+\beta\geq 0$, $0<\alpha<d$, and $\alpha-\gamma-\beta=d(\frac{1}{p}-\frac{1}{q})$.

\textup{(a)} \, If \, $1< p\le q<\infty$ and $\beta<\frac{d}{p'}$, then $\mathbf{c}(\alpha,\beta,\gamma,p,q,d)<\infty$.

\medskip

\textup{(b)} \, If \, $p=1$, $1<q<\infty$, $\beta\le 0$, then, for $f\in \mathcal{S}(\mathbb{R}^d)$ and
$\lambda>0$,
\[
\int_{\{x\in \mathbb{R}^d\colon |x|^{-\gamma}|I_{\alpha}f(x)|>\lambda\}}\,d\mu(x)\lesssim \Bigl(\frac{\||x|^{\beta}f(x)\|_1}{\lambda}\Bigr)^q.
\]
\end{theorem}

The part (a) in Theorem~\ref{thm1.1} was proved by G.H.~Hardy and
J.E.~Littlewood \cite{HarLit28} for $d=1$, S.~Sobolev \cite{Sob38} for $d>1$
and $\gamma=\beta=0$, E.M.~Stein and G.~Weiss \cite{SteWei58} in the general
case. The conditions for weak boundedness can be found in \cite{Saw, GadGul11}.

The sharp constant $\mathbf{c}(\alpha,0,0,p,q,d)$ in the non-weighted Sobolev
inequality was calculated by E.H. Lieb \cite{Lie83} in any of the following
cases:  (1) $q=p'$, $1<p<2$, (2) $q=2$, $1<p<2$,  (3) $p=2$, $2<q<\infty$.
Moreover, in these cases there exist maximazing functions. In the weighted
Hardy--Littlewood--Sobolev inequality the constant
$\mathbf{c}(\alpha,\beta,\gamma,p,q,d)$ is known only for $q=p$.

\begin{theorem}\label{thm1.2}
If \, $d\in \mathbb{N}$, $1<p<\infty$, $\gamma<\frac{d}{p}$, $\beta<\frac{d}{p'}$, $\alpha>0$, and $\gamma=\alpha-\beta$, then
\[
 \mathbf{c}(\alpha,\beta,\gamma,p,p,d)=2^{-\alpha}\frac{\Gamma(\frac{1}{2}(\frac{d}{p}-\alpha+\beta))
\Gamma(\frac{1}{2}(\frac{d}{p'}-\beta))}
{\Gamma(\frac{1}{2}(\frac{d}{p'}+\alpha-\beta))
\Gamma(\frac{1}{2}(\frac{d}{p}+\beta))}.
\]
\end{theorem}

Theorem~\ref{thm1.2} was proved by I.W.~Herbst \cite{Her77} for $\beta=0$ and   W.~Beckner \cite{Bec08} and S.~Samko \cite{Sam05} in the general case.



For $\alpha\in\mathbb{R}$, we define the Riesz potential in the distributional sense. Let $\Phi$ be the Lizorkin space \cite{Liz63}, \cite[p.~39]{Sam02}, that is, a subspace of the
Schwartz space $\mathcal{S}(\mathbb{R}^{d})$ which consists of functions orthogonal to all
polynomials:
\[
\int_{\mathbb{R}^{d}}x^nf(x)\,d\mu(x)=0,\quad n=(n_1,\dots,n_d)\in \mathbb{Z}^d_+.
\]
The subspace $\Phi$  is invariant with respect to the operator $I_{\alpha}$ and its
inverse $I_{\alpha}^{-1}=(-\Delta)^{\alpha/2}$:
\[
I_{\alpha}(\Phi)=(-\Delta)^{\alpha/2}(\Phi)=\Phi,
\]
where  $\Delta$ is the Laplacian.
 Note  that $\Phi$ is dense in $L_p(\mathbb{R}^d,  |x|^{\beta p}\,d\mu)$
for $1<p<\infty$ and $\beta\in (-d/p,\, d/p')$ \cite[p. 41]{Sam02}.

It is worth mentioning that the Stein--Weisz inequality \eqref{eq1.1} on $\Phi$ is equivalent to the
Hardy--Rellich inequality
\[
\big\| |x|^{-\gamma}f(x)\big\|_{q}\le
\mathbf{c}(\alpha,\beta,\gamma,p,q,d)\big\||x|^{\beta}(-\Delta)^{\alpha/2}f(x)\big\|_{p}.
\]

Let $D_jf(x)$ be the usual partial derivative with respect to a variable $x_j$, $j=1,\dots,d$,
$D=(D_1,\dots,D_d)$, $D^nf(x)=\prod_{j=1}^dD_j^{n_j}f(x)$, $n\in \mathbb{Z}^d_+$. The subspace
\[\Psi=\{\mathcal{F}(f)\colon f\in\Phi\}=\{f\in\mathcal{S}(\mathbb{R}^{d})\colon D^nf(0)=0,\,\, n\in \mathbb{Z}^d_+\}
\]
is invariant with respect to the operator $\mathcal{F}(I_{\alpha})$ and $\mathcal{F}((-\Delta)^{\alpha/2})$:
\[
\mathcal{F}(I_{\alpha})(\Psi)=\mathcal{F}((-\Delta)^{\alpha/2})(\Psi)=\Psi.
\]
For a distribution $f\in\Phi'$ and $\alpha\in \mathbb{R}$ we set
\[
I_{\alpha}f=\mathcal{F}^{-1}|\,{\cdot}\,|^{-\alpha}\mathcal{F}(f),\quad (-\Delta)^{\alpha/2}f=\mathcal{F}^{-1}|\,{\cdot}\,|^{\alpha}\mathcal{F}(f).
\]
If $\varphi\in\Phi$, then
\[
\<I_{\alpha}f,\varphi\>=\<f,\mathcal{F}|\,{\cdot}\,|^{-\alpha}\mathcal{F}^{-1}(\varphi)\>,\quad
\<(-\Delta)^{\alpha/2}f,\varphi\>=\<f,\mathcal{F}|\,{\cdot}\,|^{\alpha}\mathcal{F}^{-1}(\varphi)\>.
\]

One of the generalizations of the Fourier transform is the Dunkl transform
$\mathcal{F}_{k}$ (see \cite{Dun92,Ros02}). Our main goal in this paper is to prove analogues of Theorems~\ref{thm1.1} and \ref{thm1.2} for the Riesz potential associated with the Dunkl transform. We shall call it {\it the D-Riesz potential}.

Let a finite subset $R\subset \mathbb{R}^{d}\setminus\{0\}$ be a root system,
let $R_{+}$ be positive subsystem of $R$, let $G(R)\subset O(d)$ be finite reflection group,
generated by reflections $\{\sigma_{a}\colon a\in R\}$, where $\sigma_{a}$ is a
reflection with respect to hyperplane $\<a,x\>=0$, let $k\colon R\to
\mathbb{R}_{+}$ be $G$-invariant multiplicity function. Recall that a finite
subset $R\subset \mathbb{R}^{d}\setminus\{0\}$ is called a root system, if
\[
R\cap\mathbb{R} a=\{a, -a\}\quad \text{and}\quad \sigma_{a}R=R\ \text{for all }a\in R.
\]

Let
\[
v_{k}(x)=\prod_{a\in R_{+}}|\<a,x\>|^{2k(a)}
\]
be the Dunkl weight.
The normalized Macdonald--Metha--Selberg constant is given by
\[
c_{k}^{-1}=\int_{\mathbb{R}^{d}}e^{-|x|^{2}/2}v_{k}(x)\,dx.
\]
Let
$L^{p}(\mathbb{R}^{d},d\mu_{k})$  be the space of complex-valued Lebesgue measurable functions $f$ such that
\[
\|f\|_{p,d\mu_{k}}=\Bigl(\int_{\mathbb{R}^{d}}|f|^{p}\,d\mu_{k}\Bigr)^{1/p}<\infty,
\]
where
$d\mu_{k}(x)=c_{k}v_{k}(x)dx$ is the Dunkl measure.
Assume that
\begin{equation}\label{eq1.2-}
T_{j}f(x)=D_{j}f(x)+
\sum_{a\in R_{+}}k(a)\<a,e_{j}\>\,\frac{f(x)-f(\sigma_{a}x)}{\<a,x\>}
\end{equation}
are differential-differences Dunkl operators, $j=1,\ldots,d$, and
$\Delta_k=\sum_{j=1}^dT_j^2$ is the Dunkl Laplacian.

The Dunkl kernel $E_{k}(x, y)$ is a unique
solution of the system
\[
T_{j}f(x)=y_{j}f(x),\quad j=1,\ldots,d,\qquad f(0)=1.
\]
Let $e_k(x,y)=E_{k}(x, iy)$. It plays the role of a generalized exponential function. Its properties are
similar to those of the classical exponential function $e^{i\<x, y\>}$. Several basic properties
follow from the integral representation given by M.~R\"{o}sler
\cite{Ros99}
\begin{equation}\label{eq1.2}
e_k(x,y)=\int_{\mathbb{R}^d}e^{i\<\xi,y\>}\,d\mu_x^k(\xi),
\end{equation}
where $\mu_x^k$ is a probability Borel measure, whose support is contained in
$\mathrm{co}{}(\{gx\colon g\in G(R)\})$ the convex hull of the $G$-orbit of $x$
in $\mathbb{R}^d$. In particular, $|e_k(x,y)|\leq 1$ and $\supp
\mu_x^k\subset B_{|x|}$,
where $B_r$ is the Euclidean ball of radius $r$ centered at 0.

For $f\in L^{1}(\mathbb{R}^{d},d\mu_{k})$, the Dunkl transform is defined by the  equality
\[
\mathcal{F}_{k}(f)(y)=\int_{\mathbb{R}^{d}}f(x)\overline{e_{k}(x,
y)}\,d\mu_{k}(x).
\]
If $k\equiv0$, then $\mathcal{F}_{0}$ is the Fourier transform $\mathcal{F}$. We note that $\mathcal{F}_{k}(e^{-|\,\cdot\,|^2/2})(y)=e^{-|x|^2/2}$
and  $\mathcal{F}_{k}^{-1}(f)(x)=\mathcal{F}_{k}(f)(-x)$.
The Dunkl transform is isometry in $\mathcal{S}(\mathbb{R}^d)$ and $L^{2}(\mathbb{R}^{d},d\mu_{k})$ and 
$\|f\|_{2, d\mu_{k}}=\|\mathcal{F}_{k}(f)\|_{2, d\mu_{k}}$.

M. R\"{o}sler \cite{Ros98} defined the generalized translation operator $\tau^y$, $y\in \mathbb{R}^d$, on $L^{2}(\mathbb{R}^{d},d\mu_{k})$ by equality
\[
\mathcal{F}_k(\tau^yf)(z)=e_k(y, z)\mathcal{F}_k(f)(z),
\]
or
\begin{equation}\label{eq1.3}
\tau^yf(x)=\int_{\mathbb{R}^d}e_k(y, z)e_k(x, z)\mathcal{F}_k(f)(z)\,d\mu_k(z).
\end{equation}
It acts from $L^{2}(\mathbb{R}^{d},d\mu_{k})$ to $L^{2}(\mathbb{R}^{d},d\mu_{k})$ and $\|\tau^y\|_{2\to 2}=1$.

If $k\equiv0$, then $\tau^{y}f(x)=f(x+y)$.
If $f\in \mathcal{S}(\mathbb{R}^d)$, then $\tau^{y}f(x)\in \mathcal{S}(\mathbb{R}^d)\times\mathcal{S}(\mathbb{R}^d)$ and  equality \eqref{eq1.3} holds pointwise. K. Trim\`{e}che extended $\tau^y$ on $C^{\infty}(\mathbb{R}^d)$ \cite{Tri02}. For example, $\tau^y1=1$. In general, $\tau^{y}$ is not positive operator and the question of its $L_p$-boundedness remains open.

First, we define the D-Riesz potential for distributions. Let
\[
\Phi_k=\Bigl\{f\in \mathcal{S}(\mathbb{R}^d)\colon \int_{\mathbb{R}^{d}}x^nf(x)\,d\mu_k(x)=0,\quad n\in \mathbb{Z}^d_+\Bigr\}
\]
be the weighted Lizorkin space,
\[
\Psi_k=\{\mathcal{F}_{k}(f)\colon f\in\Phi_k\}. 
\]

For $\alpha\in\mathbb{R}$, we define the D-Riesz potential on $\Phi_k$ by equality \[ I^{k}_{\alpha}f=\mathcal{F}_{k}^{-1}|\,{\cdot}\,|^{-\alpha}\mathcal{F}_{k}(f). \]

In Section 6,
we will prove that $\Psi_k=\Psi$ (see Theorem \ref{thm6.1}) and then  $I^{k}_{\alpha}(\Phi_{k})=\Phi_{k}$ and $\mathcal{F}_{k}(I_{\alpha}^{k})(\Psi_{k})=\Psi_{k}$.
Therefore, we can define the D-Riesz potential $I^{k}_{\alpha}$ for $f\in\Phi_k'$ and $\alpha\in \mathbb{R}$  by the same equality
$I^{k}_{\alpha}f=\mathcal{F}_{k}^{-1}|\,{\cdot}\,|^{-\alpha}\mathcal{F}_{k}(f)$ as follows
\[
\<I_{\alpha}f,\varphi\>=\<f,\mathcal{F}_{k}|\,{\cdot}\,|^{-\alpha}\mathcal{F}^{-1}_{k}(\varphi)\>,\quad\varphi\in\Phi_k.
\]
We will also prove (see Theorem \ref{thm5.2}) that $\Phi_k$ is dense in $L_p(\mathbb{R}^d,  |x|^{\beta p}\,d\mu_k)$
for $1<p<\infty$ and $\beta\in (-d_k/p,\, d_k/p')$, where
\begin{equation}\label{eq1.4}
d_k=2\lambda_k+2,\quad \lambda_k=\frac{d}{2}-1+\sum_{a\in R_+}k(a).
\end{equation}

S.~Thangavelu and Y.~Xu  defined \cite {ThaXu07} the D-Riesz potential on Schwartz space as follows 
\begin{equation}\label{eq1.5}
I_{\alpha}^kf(x)=(\gamma^k_{\alpha})^{-1}\int_{\mathbb{R}^d}\tau^{-y}f(x)|y|^{\alpha-d_k}\,d\mu_k(y),
\end{equation}
where $0<\alpha<d_k$ and $\gamma^k_{\alpha}=2^{\alpha-d_k/2}\Gamma(\alpha/2)/\Gamma((d_k-\alpha)/2)$.

We are interested  in the Stein--Weiss inequality for the D-Riesz potential
\begin{equation}\label{eq1.6}
\big\| |x|^{-\gamma}I_{\alpha}^kf(x)\big\|_{q,d\mu_{k}}\leq \mathbf{c}_k(\alpha,\beta,\gamma,p,q,d)\big\||x|^{\beta}f(x)\big\|_{p,d\mu_{k}},\quad
f\in \mathcal{S}(\mathbb{R}^d),
\end{equation}
with the sharp constant $\mathbf{c}_k(\alpha,\beta,\gamma,p,q,d)$ and $1<p\le q<\infty$.
On $\Phi_k$, it is equivalent to the Hardy--Rellich type inequality
\[
\big\|
|x|^{-\gamma}f(x)\big\|_{q,d\mu_{k}}\le
\mathbf{c}_{k}(\alpha,\beta,\gamma,p,q,d)\big\||x|^{\beta}(-\Delta_{k})^{\alpha/2}f(x)\big\|_{p,d\mu_{k}}.
\]

Our main results read as follows.

\begin{theorem}\label{thm1.3}
If\, $d\in \mathbb{N}$, $1<p<\infty$, $\gamma<\frac{d_k}{p}$,
$\beta<\frac{d_k}{p'}$, $\alpha>0$, and $\alpha=\gamma+\beta$, then
\[
\begin{aligned}
\mathbf{c}_k(\alpha,\beta,\gamma,p,p,d)&=
2^{-\alpha}\,\frac{\Gamma(\frac{1}{2}(\frac{d_k}{p}-\gamma))
\Gamma(\frac{1}{2}(\frac{d_k}{p'}-\beta))}
{\Gamma(\frac{1}{2}(\frac{d_k}{p'}+\gamma))
\Gamma(\frac{1}{2}(\frac{d_k}{p}+\beta))}\\
&=\mathbf{c}(\alpha,\beta,\gamma,p,p,d_k).
\end{aligned}
\]
\end{theorem}


\begin{theorem}\label{thm1.4}
Let \, $d\in \mathbb{N}$, $1\le p\le q<\infty$,
$\gamma<\frac{d_k}{q}$, $\gamma+\beta\geq 0$, $0<\alpha<d_k$, and $\alpha-\gamma-\beta=d_k(\frac{1}{p}-\frac{1}{q})$.

\textup{(a)} \, If \, $1< p\le q<\infty$ and $\beta<\frac{d_k}{p'}$, then $\mathbf{c}_k(\alpha,\beta,\gamma,p,q,d)<\infty$.

\medskip

\textup{(b)} \, If \, $p=1$, $1<q<\infty$, $\beta\le 0$,  and
$\lambda>0$, then
\[\quad
\int_{\{x\in \mathbb{R}^d\colon |x|^{-\gamma}|I_{\alpha}^kf(x)|>\lambda\}}\,d\mu_k(x)\lesssim \Big(
{\big \||x|^{\beta}f(x)\big \|_{1,d\mu_k}}/{\lambda}\Big)^q,  \qquad f\in \mathcal{S}(\mathbb{R}^d).
\]
\end{theorem}

For $k\equiv 0$, Theorems \ref{thm1.3} and \ref{thm1.4} become
Theorems \ref{thm1.1} and \ref{thm1.2}, therefore it is enough to consider the case $k\not\equiv 0$,
i.e., $\lambda_k=\frac{d}{2}-1+\sum_{a\in R_+}k(a)>-1/2$ and
$d_k=2\lambda_k+2>1$. It is clear that $d_k$ plays the role of the
generalized dimension of the space $(\mathbb{R}^d,d\mu_k)$.

For the reflection group $\mathbb{Z}_2^d$ and $\gamma=\beta=0$, Theorem~\ref{thm1.4} was proved in \cite{ThaXu07}. For arbitrary reflection group $G$ and $\gamma=\beta=0$, it was proved by S.~Hassani, S.~Mustapha and M.~Sifi \cite{HasMusSif09}.
Following an idea from \cite{ThaXu07}, we have recently given another proof in  \cite{GorIvaTik17}.
Regarding the weighted setting, part (a) was proved in \cite{AbdLit15} in the case
$q=p$ 
 under more restrictive conditions
$1<p<\infty,$ $0<\gamma<\frac{d_k}{p}$, $0<\beta<\frac{d_k}{p'}$,  and $\alpha>0.$

To estimate the $L^p$-norm of operator $I_{\alpha}^k$, S.~Thangavelu and Y.~Xu \cite {ThaXu07} used the maximal function, defined  for $f\in \mathcal{S}(\mathbb{R}^d)$ as follows
\[
M^kf(x)=\sup\limits_{r>0}\frac{|\int_{\mathbb{R}^d}\tau^{-y}f(x)\chi_{B_r}(y)\,d\mu_k(y)|}{\int_{B_r}\,d\mu_k},
\]
where $B_r=\{x\colon |x|\le r\}$. 
  They proved the strong $L^p$-boundedness of $M^k$ for $1<p<\infty$ and the weak boundedness for $p=1$ \cite {ThaXu05}.

We will use Theorem \ref{thm1.4} to obtain weighted boundedness of the fractional maximal function $M^k_{\alpha}f$, $0\le\alpha<d_k$,
  given by
\begin{align*}
M^k_{\alpha}f(x)&=\sup\limits_{r>0}r^{\alpha-d_k}\Bigl|\int_{\mathbb{R}^d}\tau^{-y}f(x)\chi_{B_r}(y)\,d\mu_k(y)\Bigr|\\
&=\sup\limits_{r>0}r^{\alpha-d_k}\Bigl|\int_{\mathbb{R}^d}f(x)\tau^{-y}\chi_{B_r}(x)\,d\mu_k(y)\Bigr|.
\end{align*}
If $\alpha=0$, then $M^k_{0}$ coincides with  $M^k$ up to a constant.
Since  $\tau^y$ is a positive operator on radial functions \cite{Ros03,ThaXu05},
and using
\[
M^k_{\alpha}f(x)\leq M^k_{\alpha}|f|(x)\lesssim I_{\alpha}^k|f|(x),
\]
Theorem~\ref{thm1.4} implies the boundedness conditions of the fractional maximal function.

\begin{theorem}\label{thm1.5}
Let \, $d\in \mathbb{N}$, $1\le p\le q<\infty$,
$\gamma<\frac{d_k}{q}$, $\gamma+\beta\geq 0$, $0<\alpha<d_k$, $\alpha-\gamma-\beta=d_k(\frac{1}{p}-\frac{1}{q})$, and $f\in \mathcal{S}(\mathbb{R}^d)$.

\textup{(a)} \, If \, $1< p\le q<\infty$ and $\beta<\frac{d_k}{p'}$, then
\[
\big\||x|^{-\gamma}M^k_{\alpha}f(x)\big\|_{q,d\mu_{k}}\lesssim \big\| |x|^{\beta}f(x)\big\|_{p,d\mu_{k}}.
\]

\textup{(b)} \, If \, $p=1$, $1<q<\infty$, $\beta\le 0$, and $\lambda>0$, then
\[
\int_{\{x\in \mathbb{R}^d\colon |x|^{-\gamma}|M^k_{\alpha}f(x)|>\lambda\}}\,d\mu_k(x)\lesssim \Big (
{\big\||x|^{\beta}f(x)\big\|_{1,d\mu_{k}}}/{\lambda}\Big)^q.
\]
\end{theorem}

In the case $\gamma=\beta=0$ Theorem~\ref{thm1.5} was proved in \cite{HasMusSif09}.

The paper is organized as follows. In the next section, we obtain the sharp inequalities for Mellin convolution and
investigate the following representation of the Riesz potential
$$I_{\alpha}^{k}f(x)=\int_{\mathbb{R}^{d}}f(y)\Phi(x,y)\,d\mu_{k}(y)
$$
and basic properties of the kernel$$
\Phi(x,y)=\frac{2^{d_k/2-\alpha}}{\Gamma(\alpha/2)}
\int_{0}^{\infty}s^{(d_k-\alpha)/2-1}\tau^{-y}(e^{-s|\Cdot|^{2}})(x)\,ds,\quad (x,y)\neq (0,0).
$$
In Section 3, we prove sharp $(L_p,L_p)$  Hardy's inequalities with weights for the averaging  operator
$
Hf(x)=\int_{|y|\le |x|}f(y)\,d\mu_k(y).
$
In the classical setting ($k=0$), this result was proved by M.~Christ and L.
Grafakos~\cite{ChrGra95} and Z.W.~Fu, L.~Grafakos, S.Z.~Lu and
F.Y.~Zhao~\cite{FuGraLuZha12}. Sections 4 and 5 are devoted to the proofs of
Theorems \ref{thm1.3} and \ref{thm1.4} corespondingly. We finish with Section 6, which contains
some important properties of the spaces $\Phi_k
$
and
$\Psi_k$.

\section{Notations and auxiliary statements}
Set as usual
 $\mathbb{R}_{+}=[0,\infty)$,
$\mathbb{S}^{d-1}=\{x\in \mathbb{R}^{d}\colon |x|=1\}$ and  $x=rx'\in \mathbb{R}^{d}$, $r=|x|\in \mathbb{R}_{+}$, $x'\in \mathbb{S}^{d-1}$.
Let
\[
d\nu(r)=r^{-1}\,dr,\quad d\nu_{\lambda}(r)=b_{\lambda}r^{2\lambda+1}\,dr,\quad b_{\lambda}^{-1}=2^{\lambda}\Gamma(\lambda+1),\quad \lambda\geq-1/2,
\]
be the measures on $\mathbb{R}_{+}$,  $$d\sigma_{k}(x')=a_{k}v_{k}(x')\,dx'$$ be the probability
measure on $\mathbb{S}^{d-1}$,  and
\[
d\mu_{k}(x)=c_{k}v_{k}(x)\,dx,\quad dm_{k}(x)=d\nu(r)\,d\sigma_{k}(x')
\]
be the measures on $\mathbb{R}^{d}$.

Note that
\begin{equation}\label{eq2.1}
d\mu_{k}(x)=d\nu_{\lambda_{k}}(r)\,d\sigma_{k}(x')=b_{\lambda_{k}}|x|^{d_k}\,dm_{k}(x),
\end{equation}
where $\lambda_k$ and $d_k$ are defined in \eqref{eq1.4}.

Let $L^{p}(X,d\mu)$, $1\leq p\leq \infty$, be the Banach space with the norm
\[
\|f\|_{p,d\mu}=\begin{cases}
\Bigl(\int_{X}|f|^{p}\,d\mu\Bigr)^{1/p},&p<\infty,\\
\supvrai_{X}|f|,&p=\infty.
\end{cases}
\]

Depending on the context, we assume that $L^{p}=L^{p}(X,d\mu)$ and 
$\|f\|_{p}=\|f\|_{p,d\mu}$.

\subsection{Convolution inequalities.}
The Mellin convolution is given by
\[
A_{g}f(r)=(f*g)(r)=\int_{0}^{\infty}f(r/t)g(t)\,d\nu(t).
\]
We will frequently use the fact that
\begin{equation}\label{eq2.2}
(f*g)(r)=(g*f)(r).
\end{equation}

\begin{lemma}\label{lem2.1}
Let $L^{p}=L^{p}(\mathbb{R}_{+}, d\nu)$, $1\leq p\leq\infty$. If $f\in L^{p}$, $h\in L^{p'}$, $g\in L^{1}$, then
\[
\|f*g\|_{p}\leq \|g\|_{1}\|f\|_{p},
\]
or
\begin{equation*}
\Bigl|\int_{0}^{\infty}\int_{0}^{\infty}h(r)f(t)g(r/t)\,d\nu(t)\,d\nu(r)\Bigr|\leq
\|g\|_{1}\|h\|_{p'}\|f\|_{p}.
\end{equation*}
If $g\geq 0$, then
\begin{equation}\label{eq2.4}
\|A_{g}\|_{p\to p}=\|g\|_{1},
\end{equation}
or
\[
\sup_{\|f\|_{p}\leq 1}\,\sup_{\|h\|_{p'}\leq 1}\,
\Bigl|\int_{0}^{\infty}h(r)A_{g}f(r)\,d\nu(r)\Bigr|=\|g\|_{1}.
\]
\end{lemma}

\begin{proof}
For the classical convolution on the $\mathbb{R}$, see, e.g.,
\cite{Her77}.  We sketch  the proof here for completeness of the {exposition.}
For $1<p<\infty$, using H\"{o}lder's inequality, we obtain
\[
\Bigl|\int_{0}^{\infty}f(r/t)g(t)\,d\nu(t)\Bigr|\leq \Bigl(\int_{0}^{\infty}|f(r/t)|^p\,|g(t)|\,d\nu(t)\Bigr)^{1/p}\Bigl(\int_{0}^{\infty}|g(t)|\,d\nu(t)\Bigr)^{1/p'}
\]
and
\begin{align*}
\|f*g\|_{p}&\leq \Bigl(\int_{0}^{\infty}\int_{0}^{\infty}|f(r/t)|^p\,|g(t)|\,d\nu(t)\,d\nu(r)\Bigr)^{1/p}\|g\|_1^{1/p'}\\
&=\Bigl(\int_{0}^{\infty}|g(t)|\int_{0}^{\infty}|f(r/t)|^p\,d\nu(r)\,d\nu(t)\Bigr)^{1/p}\|g\|_1^{1/p'}
=\|f\|_{p}\,\|g\|_1.
\end{align*}

Let $g\ge 0$.
If $p=1$, $f\in L^1$, $f\ge 0$, then
\begin{align*}
\|A_{g}f\|_1&=\int_{0}^{\infty}\int_{0}^{\infty}f(r/t)g(t)\,d\nu(t)\,d\nu(r)\\
&=\int_{0}^{\infty}f(r)\,d\nu(r)\int_{0}^{\infty}g(t)\,d\nu(t)=\|g\|_1\|f\|_1,
\end{align*}
which gives \eqref{eq2.4}.
If $p=\infty$, we define  $f=\chi_{[\lambda,1/\lambda]}$, $0<\lambda<1$. Then $\|f\|_{\infty}=1$ and, for $r\in[1,2]$,
\begin{align*}
A_{g}f(r)&=\int_{0}^{\infty}f(t)g(r/t)\,d\nu(t)=\int_{\lambda}^{1/\lambda}g(r/t)\,d\nu(t)\\
&=\int_{r\lambda}^{r/\lambda}g(t)\,d\nu(t)\ge\int_{2\lambda}^{1/\lambda}g(t)\,d\nu(t)\to\|g\|_1,\qquad \lambda\to 0.
\end{align*}
If $1<p<\infty$, $f=(2\lambda)^{-1/p}\chi_{[e^{-\lambda}, e^{\lambda}]}$, and $h=(2\lambda)^{-1/p'}\chi_{[e^{-\lambda}, e^{\lambda}]}$, then $\|f\|_{p}=\|h\|_{p'}=1$ and
by the Lebesgue dominated convergence theorem
\begin{align*}
\|A_{g}\|_{p\to p}&\geq\lim\limits_{\lambda\to\infty}\Bigl\{(2\lambda)^{-1}
\int_{e^{-\lambda}}^{e^{\lambda}}\int_{e^{-\lambda}}^{e^{\lambda}}g(r/t)\,d\nu(r)\,d\nu(t)\Bigr\}\\
&=\lim\limits_{\lambda\to\infty}\Bigl\{(2\lambda)^{-1}
\int_{e^{-\lambda}}^{e^{\lambda}}\int_{e^{-\lambda}/t}^{e^{\lambda}/t}g(r)\,d\nu(r)\,d\nu(t)\Bigr\}\\
&=\lim\limits_{\lambda\to\infty}(2\lambda)^{-1}
\Bigl\{\int_{e^{-2\lambda}}^{1}\int_{e^{-\lambda}/r}^{e^{\lambda}}
\,\frac{dt}{t}\,g(r)\,\frac{dr}{r}+\int_{1}^{e^{2\lambda}}\int_{e^{-\lambda}}^{e^{\lambda}/r}
\,\frac{dt}{t}\,g(r)\,\frac{dr}{r}\Bigr\}\\
&=\lim\limits_{\lambda\to\infty}(2\lambda)^{-1}
\Bigl\{\int_{e^{-2\lambda}}^{1}g(r)(2\lambda+\ln r)\,\frac{dr}{r}+
\int_{1}^{e^{2\lambda}}g(r)(2\lambda-\ln r)\,\frac{dr}{r}\Bigr\}\\
&=\int_{0}^{\infty}g(r)\,\frac{dr}{r}=\|g\|_1.
\qedhere\end{align*}
\end{proof}

\subsection{
A representation of the Riesz potential.}
We will use the following representation (see \cite{AbdLit15}), which is  different from
the definition (\ref{eq1.5}):
\begin{equation}\label{eq2.7}
I_{\alpha}^{k}f(x)=\int_{\mathbb{R}^{d}}f(y)\Phi(x,y)\,d\mu_{k}(y),
\end{equation}
where 
\begin{equation}\label{eq2.8}
\Phi(x,y)=\frac{2^{d_k/2-\alpha}}{\Gamma(\alpha/2)}
\int_{0}^{\infty}s^{(d_k-\alpha)/2-1}\tau^{-y}(e^{-s|\Cdot|^{2}})(x)\,ds,\quad (x,y)\neq (0,0).
\end{equation}
To verify \eqref{eq2.7}, we first remark that
the convolution
\[
(f\ast {}_{k}g)(x)=\int_{\mathbb{R}^{d}}\tau^{-y}f(x)\,g(y)\,d\mu_{k}(y)
\]
is commutative, i.e.,
$
(f\ast {}_{k}g)(x)=(g\ast {}_{k}f)(x).
$ 
Indeed, we have the following

\begin{lemma}\label{lem2.2}
If $f\in\mathcal{S}(\mathbb{R}^d)$, $g\in L^{1}(\mathbb{R}^{d},d\mu_{k})$, and $f_t(x)=f(tx)$, then
\begin{equation}\label{eq2.5}
\int_{\mathbb{R}^{d}}\tau^{-y}f(x)\,g(y)\,d\mu_{k}(y)=
\int_{\mathbb{R}^{d}}f(y)\tau^{-y}g(x)\,d\mu_{k}(y),
\end{equation}
\begin{equation}\label{eq2.6}
\mathcal{F}_k(f_t)(z)=\frac{1}{t^{d_k}}\mathcal{F}_k(f)\Bigl(\frac{z}{t}\Bigr),\quad \tau^{y}(f_t)(x)=\tau^{ty}f(tx).
\end{equation}
\end{lemma}

Relation \eqref{eq2.5} has been recently proved in \cite{GorIvaTik17}.
Equalities \eqref{eq2.6} can be verified by simple calculations.

\begin{remark}\label{rem2.1}
It is worth mentioning that if  the convolution is defined by
\[
(f\ast {}_{k}g)(x)=\int_{\mathbb{R}^{d}}\tau^{x}f(y)\,g(y)\,d\mu_{k}(y)
\]
(see \cite{Ros03}), it is not commutative:
\[
\int_{\mathbb{R}^{d}}\tau^{x}f(y)\,g(y)\,d\mu_{k}(y)=
\int_{\mathbb{R}^{d}}f(y)\tau^{-x}g(y)\,d\mu_{k}(y).
\]
\end{remark}

Completing   the proof of \eqref{eq2.7}, we use
 \eqref{eq1.5}, \eqref{eq2.5} and the fact that  (see \cite{ThaXu07})
\begin{equation}\label{eq2zv}
\frac{1}{|y|^{d_k-\alpha}}=\frac{1}{\Gamma((d_k-\alpha)/2)}
\int_{0}^{\infty}s^{(d_k-\alpha)/2-1}e^{-s|y|^{2}}\,ds,
\end{equation}
to obtain
\begin{align*}
I_{\alpha}^kf(x)&=(\gamma_k^{\alpha})^{-1}\int_{\mathbb{R}^d}\tau^{-y}f(x)|y|^{\alpha-d_k}\,d\mu_k(y)\\
&=\frac{2^{d_k/2-\alpha}}{\Gamma(\alpha/2)}\int_{\mathbb{R}^d}\tau^{-y}f(x)\int_{0}^{\infty}s^{(d_k-\alpha)/2-1}e^{-s|y|^{2}}\,ds\,d\mu_k(y)\\
&=\frac{2^{d_k/2-\alpha}}{\Gamma(\alpha/2)}\int_{0}^{\infty}s^{(d_k-\alpha)/2-1}\int_{\mathbb{R}^d}\tau^{-y}f(x)e^{-s|y|^{2}}\,d\mu_k(y)\,ds\\
&=\frac{2^{d_k/2-\alpha}}{\Gamma(\alpha/2)}\int_{\mathbb{R}^d}f(y)\int_{0}^{\infty}s^{(d_k-\alpha)/2-1}\tau^{-y}(e^{-s|\,{\cdot}\,|^{2}})(x)\,ds\,d\mu_k(y).
\end{align*}
{The interchange of the order of integration
 is legitimate}, since, for any $x\in\mathbb{R}^d$, the  iterated integral
\[
\int_{\mathbb{R}^d}\left|\tau^{-y}f(x)\right|\int_{0}^{\infty}s^{(d_k-\alpha)/2-1}e^{-s|y|^{2}}\,ds\,d\mu_k(y)
\]
converges, where we have used the fact that
 $\tau^{y}f(x)\in \mathcal{S}(\mathbb{R}^d)\times\mathcal{S}(\mathbb{R}^d)$
 whenever
 $f\in \mathcal{S}(\mathbb{R}^d)$.

\subsection{Properties of the kernel $\Phi(x,y)$.}
We will need the following notation.
Let $\lambda\ge -1/2$, $J_{\lambda}(t)$ be the classical Bessel function of degree $\lambda$ and
\[
j_{\lambda}(t)=2^{\lambda}\Gamma(\lambda+1)t^{-\lambda}J_{\lambda}(t)
\]
be the normalized Bessel function. The Hankel transform is defined as follows
\[
\mathcal{H}_{\lambda}(f_0)(r)=\int_{0}^{\infty}f_0(t)j_{\lambda}(rt)\,d\nu_{\lambda}(t),\quad r\in \mathbb{R}_{+}.
\]
It is a unitary operator in $L^{2}(\mathbb{R}_{+},d\nu_{\lambda})$ and
$\mathcal{H}_{\lambda}^{-1}=\mathcal{H}_{\lambda}$ \cite[Chap.~7]{BatErd53}.
If $\lambda=\lambda_k$, the Hankel transform is a restriction of the
Dunkl transform on radial functions.
Recall that we assume that $\lambda_k>-1/2$.

For $\lambda>-1/2$, let us consider the Gegenbauer-type translation operator (see, e.g., \cite{Pla07})
\begin{equation}\label{eq2.9}
G^sf_0(r)=c_\lambda\int_{0}^{\pi}f_0(\sqrt{r^2+s^2-2rs\cos\varphi})\sin^{2\lambda}\varphi\,d\varphi, \end{equation}
where $c_\lambda=\frac{\Gamma(\lambda+1)}{\Gamma(1/2)\Gamma(\lambda+1/2)}$. If $f_0\in \mathcal{S}(\mathbb{R}_+)$, then
\begin{equation}\label{eq2.10}
G^sf_0(r)=\int_{0}^{\infty}j_{\lambda}(rt)j_{\lambda}(st)\mathcal{H}_{\lambda}(f_0)(t)\,d\nu_{\lambda}(t).
\end{equation}
We will also need the following partial case of the Funk--Hecke formula \cite{Xu00}
\begin{equation}\label{eq2.11}
\int_{\mathbb{S}^{d-1}}e_{k}(x, ty')\,d\sigma_k(y')=j_{\lambda_k}(t|x|).
\end{equation}

Let $x=rx'$, $y=ty'$, $r,t>\mathbb{R}_+$, and $x',y'\in\mathbb{S}^{d-1}$.

\begin{lemma}\label{lem2.3} The kernel\,\, $\Phi(x,y)$  satisfies the following properties
\begin{enumerate}

\item \quad $\Phi(x,y)=\Phi(y,x)$;

\medskip

\item\quad $\Phi(rx',ty')=r^{\alpha-d_k}\Phi(x',(t/r)y')$;

\medskip

\item \quad$\int_{\mathbb{S}^{d-1}}\Phi(rx',ty')\,d\sigma_{k}(x')=\Phi_{0}(r,t)$,
where
\noindent\quad  $$\Phi_{0}(r,t):=(\gamma^{k}_{\alpha})^{-1}c_{\lambda_{k}}
\int_{0}^{\pi}\bigl(r^{2}+t^{2}-2rt\cos \varphi\bigr)^{(\alpha-d_k)/2}
\sin^{d_k-2}\varphi\,d\varphi;$$

\medskip

\item \quad $\Phi(x,y)=(\gamma^{k}_{\alpha})^{-1}
 \tau^{-y}(|\,{\cdot}\,|^{\alpha-d_k})(x)$
or, equivalently,
$$\Phi(x,y)=(\gamma^{k}_{\alpha})^{-1}
\int_{\mathbb{R}^{d}}(|x|^{2}+|y|^{2}-2\<y,\eta\>)^{(\alpha-d_k)/2}
\,d\mu_{x}^{k}(\eta),$$

\noindent \quad where $\mu_{x}^{k}$ is a probability measure from \eqref{eq1.2}.
\end{enumerate}
\end{lemma}

\begin{proof} Recall that  $E_k(x,y)$ is the Dunkl kernel.
Using $E_{k}(\lambda x, y)=E_{k}(x, \lambda  y),$ $\lambda \in \mathbb{C}$, we have from
\cite[Sec. 4.9]{Ros98} that
$$
\int_{\mathbb{R}^{d}}e_k(x,z)e_k(-y,z)e^{-|z|^2/2}\,d\mu_k(z)=e^{-\frac{|x|^2+|y|^2}{2}}E_k(x,y).
$$
This,   \eqref{eq2.6},  and  the fact that  $\mathcal{F}_{k}(e^{-|\,\cdot\,|^2/2})(y)=e^{-|x|^2/2}$
imply that \[
\tau^{-y}(e^{-s|\,{\cdot}\,|^{2}})(x)=e^{-s(|x|^2+|y|^2)}
E_k(\sqrt{2s}\,x,\sqrt{2s}\,y).
\]
Since $E_k(x,y)=E_k(y,x)$,  the property (1) follows and,
moreover,
$$
\Phi(x,y)=\frac{2^{d_k/2-\alpha}}{\Gamma(\alpha/2)}
\int_{0}^{\infty}s^{(d_k-\alpha)/2-1}e^{-s(|x|^2+|y|^2)}
E_k(\sqrt{2s}\,x,\sqrt{2s}\,y)\,ds.
$$
Changing variables  $s\to u/r^2$, we obtain the property (2):
\begin{align*}
\Phi(rx',ty')&=r^{\alpha-d_k}\int_{0}^{\infty}u^{(d_k-\alpha)/2-1}e^{-u(1+(t/r)^2)}
E_k(\sqrt{2u}\,x',\sqrt{2u}\,(t/r)y')\,du
\\
&=r^{\alpha-d_k}\Phi(x',(t/r)y').
\end{align*}

Since, by \eqref{eq2.6} and \eqref{eq1.3}, we have
\[
\tau^{-ty'}(e^{-s|\Cdot|^{2}})(rx')=\int_{\mathbb{R}^{d}}e_k(\sqrt{2s}rx',z)e_k(-\sqrt{2s}ty',z)e^{-|z|^2/2}\,d\mu_k(z),
\]
then taking into account \eqref{eq2.1}, \eqref{eq2.9}, \eqref{eq2.10}, and \eqref{eq2.11},
we obtain
\begin{align*}
&\int_{\mathbb{S}^{d-1}}\tau^{-ty'}(e^{-s|\Cdot|^{2}})(rx')\,d\sigma_k(x')
\\
&=\int_{\mathbb{R}^{d}}e_k(-\sqrt{2s}ty',z)e^{-|z|^2/2}\int_{\mathbb{S}^{d-1}}e_k(\sqrt{2s}rx',z)\,d\sigma_k(x')\,d\mu_k(z)\\
&=\int_{0}^{\infty}j_{\lambda_k}(\sqrt{2s}ru)e^{-u^2/2}\int_{\mathbb{S}^{d-1}}e_k(-\sqrt{2s}ty',uz')\,d\sigma_k(z')\,d\nu_{\lambda_k}(u)\\
&=\int_{0}^{\infty}j_{\lambda_k}(\sqrt{2s}ru)j_{\lambda_k}(\sqrt{2s}tu)e^{-u^2/2}\,d\nu_{\lambda_k}(u)\\
&=c_{\lambda_{k}}\int_{0}^{\pi}e^{-s(r^{2}+t^{2}-2rt\cos\varphi)}\sin^{d_k-2}\varphi\,d\varphi.
\end{align*}
This and \eqref{eq2.8} imply that
\begin{align*}
&\int_{\mathbb{S}^{d-1}}\Phi(rx',ty')\,d\sigma_{k}(x')
\\
&=\frac{2^{d_k/2-\alpha}}{\Gamma(\alpha/2)}c_{\lambda_{k}}
\int_{0}^{\infty}s^{(d_k-\alpha)/2-1}\int_{0}^{\pi}e^{-s(r^{2}+t^{2}-2rt\cos\varphi)}\sin^{d_k-2}\varphi\,d\varphi\,ds.
\end{align*}
Finally, applying \eqref{eq2zv} gives
\begin{align*}
&\int_{\mathbb{S}^{d-1}}\Phi(rx',ty')\,d\sigma_{k}(x')
\\
&=(\gamma^{k}_{\alpha})^{-1}c_{\lambda_{k}}
\int_{0}^{\pi}\bigl(r^{2}+t^{2}-2rt\cos \varphi\bigr)^{(\alpha-d_k)/2}
\sin^{d_k-2}\varphi\,d\varphi=\Phi_0(r,t),
\end{align*}
i.e.,  the property (3) follows.

Let us prove the property (4). Since for radial functions $f(x)=f_{0}(|x|)\in \mathcal{S}(\mathbb{R}^{d})$ \cite{Ros03,ThaXu05}
\[
\tau^{-y}f(x)=\int_{\mathbb{R}^{d}}f_{0}(\sqrt{|x|^{2}+|y|^{2}-2\<y,\eta\>})\,
d\mu_{x}^{k}(\eta),
\]
where $\mu_{x}^{k}$ is the probability measure in \eqref{eq1.2}, we derive
\[
\tau^{-y}(e^{-s|\Cdot|^{2}})(x)=\int_{\mathbb{R}^{d}}
e^{-s(|x|^{2}+|y|^{2}-2\<y,\eta\>)}\,d\mu_{x}^{k}(\eta)
\]
and
\begin{align*}
\Phi(x,y)&=\frac{2^{d_k/2-\alpha}}{\Gamma(\alpha/2)}
\int_{0}^{\infty}s^{(d_k-\alpha)/2-1}\int_{\mathbb{R}^{d}}
e^{-s(|x|^{2}+|y|^{2}-2\<y,\eta\>)}\,d\mu_{x}^{k}(\eta)\,ds\\
&=\frac{2^{d_k/2-\alpha}}{\Gamma(\alpha/2)}
\int_{\mathbb{R}^{d}}\int_{0}^{\infty}s^{(d_k-\alpha)/2-1}
e^{-s(|x|^{2}+|y|^{2}-2\<y,\eta\>)}\,ds\,d\mu_{x}^{k}(\eta)\\
&=(\gamma^{k}_{\alpha})^{-1}
\int_{\mathbb{R}^{d}}(|x|^{2}+|y|^{2}-2\<y,\eta\>)^{(\alpha-d_k)/2}
\,d\mu_{x}^{k}(\eta),
\end{align*}
where
we have used 
 the Tonelli--Fubini Theorem for nonnegative functions.

\end{proof}

\section{Sharp Hardy's inequalities}

Define  the Hardy and Bellman operators as follows
\[
Hf(x)=\int_{|y|\le |x|}f(y)\,d\mu_k(y)
\]
and
\[
B f(x)=\int_{|y|\ge |x|}f(y)\,d\mu_k(y).
\]
Let
 $1\le p\le\infty$.
We are interested  in the weighted Hardy inequalities of the form
\begin{equation}\label{eq3.1}
\big\||x|^{-a}Hf(x)\big\|_{p,d\mu_{k}}\leq \textbf{c}_k^H(a,b,p,d)\big\||x|^{b}f(x)\big\|_{p,d\mu_{k}}
\end{equation}
and
\begin{equation}\label{eq3.2}
\big\||x|^{-a}Bf(x)\big\|_{p,d\mu_{k}}\leq \textbf{c}_k^B(a,b,p,d)\big\||x|^{b}f(x)\big\|_{p,d\mu_{k}}
\end{equation}
with the sharp constants
$ {\textbf{c}}_k^H(a,b,p,d)$ and
$ {\textbf{c}}_k^B(a,b,p,d)$.

In the classical setting ($k\equiv 0$),
the sharp constants were calculated  by M.~Christ and L. Grafakos~\cite{ChrGra95} in the non-weighted case\, ($b=0$, $a=d$)
and later by
Z.W.~Fu, L.~Grafakos, S.Z.~Lu and F.Y.~Zhao~\cite{FuGraLuZha12}\,in the general case.
 We extend these results for the Dunkl setting.
 Recall that
\[
\lambda_k=\frac{d}{2}-1+\sum_{a\in R_+}k(a),\quad d_k=2\lambda_k+2,\quad b_{\lambda_k}=\frac{1}{2^{\lambda_k}\Gamma(\lambda_k+1)}.
\]

\begin{theorem}\label{thm3.1}
Let \, $d\in \mathbb{N}$ and $1\le p\le\infty$.
Inequality \eqref{eq3.1} holds with
$\textup{{\textbf{c}}}_k^H(a,b,p,d)<\infty$ if and only if
 $\;\frac{a}{p'}>\frac{b}{p}$ and $\;a+b=d_k$. Moreover,
\[
\textup{{\textbf{c}}}_k^H(a,b,p,d)=\frac{b_{\lambda_k}}{
\frac{a}{p'}-\frac{b}{p}
}.
\]

\end{theorem}

\begin{proof} Assume that  $\;\frac{a}{p'}>\frac{b}{p}$ and $\;a+b=d_k$. We consider 
\[
\wt{H}f(x)=\int_{|y|\le |x|}|y|^{d_k/p'-b}f(y)\,dm_k(y).
\]
According to \eqref{eq2.1},  inequality \eqref{eq3.1} is equivalent to the following estimate

\[
b_{\lambda_k}\big\||x|^{-a+d_k/p}\wt{H}f(x)\big\|_{p,dm_{k}}\leq
{\textbf{c}}_k^H(a,b,p,d)\|f\|_{p,dm_{k}}.
\]
If $x=rx'$, $y=ty'$, then changing variables  $y\to (r/t)y'$ yields
\[
\wt{H}f(x)=r^{d_k/p'-b}\int_{\mathbb{R}^d}f((r/t)y')g_0(t)\,dm_k(ty'),
\]
where
\[
g_0(t)=t^{b-d_k/p'}\chi_{[1,\infty)}(t).
\]
Hence, by \eqref{eq2.2}, we have
\[
|x|^{-a+d_k/p}\wt{H}f(x)=\int_{\mathbb{R}^d}f((r/t)y')g_0(t)\,dm_k(ty')=\int_{\mathbb{R}^d}f(ty')g_0(r/t)\,dm_k(ty').
\]

Let us consider the integral
\begin{align*}
J&=\int_{\mathbb{R}^{d}}\int_{\mathbb{R}^{d}}h(rx')f(ty')g_{0}(r/t)\,dm_{k}(x)\,dm_{k}(y)\\
&=\int_{\mathbb{S}^{d-1}}\int_{\mathbb{S}^{d-1}}\int_{0}^{\infty}\int_{0}^{\infty}
h(rx')f(ty')g_{0}(r/t)\,d\nu(t)\,d\nu(r)\,d\sigma_{k}(x')\,d\sigma_{k}(y').
\end{align*}
Using  H\"{o}lder's inequality
 and Lemma~\ref{lem2.1}, we obtain
\begin{align*}
|J|&\le \int_{\mathbb{S}^{d-1}}\int_{\mathbb{S}^{d-1}}\Bigl(\int_{0}^{\infty}|h(rx')|^{p'}\,d\nu(r)\Bigr)^{1/p'}
\Bigl(\int_{0}^{\infty}|f(ty')|^p\,d\nu(t)\Bigr)^{1/p}\\
&\quad\int_{0}^{\infty}g_0(r/t)\,d\nu(t)\,d\sigma_{k}(x')\,d\sigma_{k}(y')\\
&\le \|g_0\|_1 \Bigl(\int_{\mathbb{S}^{d-1}}\int_{\mathbb{S}^{d-1}}\int_{0}^{\infty}|h(rx')|^{p'}\,d\nu(r)\,d\sigma_{k}(x')\,d\sigma_{k}(y')\Bigr)^{1/p'}\\
&\quad\Bigl(\int_{\mathbb{S}^{d-1}}\int_{\mathbb{S}^{d-1}}\int_{0}^{\infty}|f(ty')|^p\,d\nu(t)\,d\sigma_{k}(x')\,d\sigma_{k}(y')\Bigr)^{1/p}\\
&=\|g_0\|_1 \|h\|_{p',dm_k} \|f\|_{p,dm_k}.
\end{align*}
Hence,
$$
\textbf{c}_k^H(a,b,p,d)\le b_{\lambda_k}\|g_0\|_1=b_{\lambda_k}\int_{1}^{\infty}t^{b-d_k/p'}\,\frac{dt}{t}=\frac{b_{\lambda_k}}{\frac{a}{p'}-\frac{b}{p}}.
$$

Considering radial functions  $f(x)=f_0(|x|)=f_0(r)$, we note that
\[
\wt{H}f(x)=r^{d_k/p'-b}\int_{\mathbb{R}^d} f(r/t)g_0(t)\,\frac{dt}t
\]
and
\[
|x|^{a-d_k/p}\wt{H}f(x)=\int_{\mathbb{R}^d} f(r/t)g_0(t)\,\frac{dt}t.
\]
Thus,  Lemma~\ref{lem2.1} yields that
$$\textbf{c}_k^H(a,b,p,d)= b_{\lambda_k}\|g_0\|_1=\frac{b_{\lambda_k}}{\frac{a}{p'}-\frac{b}{p}}.
$$

Note that, in particular, this implies that the condition  $\;\frac{a}{p'}>\frac{b}{p}$ is necessary for
  $\textup{{\textbf{c}}}_k^H(a,b,p,d)<\infty$ to hold.
 Moreover,
if $f_t(x)=f(tx)$, then
\[
Hf_t(x)=t^{-d_k}(Hf)_t(x), \quad \big\||x|^{b}f_t(x)\big\|_{p,d\mu_{k}}=t^{-b-d_k/p}\big\||x|^{b}f(x)\big\|_{p,d\mu_{k}}
\]
and inequality \eqref{eq3.1} can be written as
\[
t^{-d_k(1+1/p)+a}\big\||x|^{-a}Hf(x)\big\|_{p,d\mu_{k}}\leq t^{-b-d_k/p}\textbf{c}_k^H(a,b,p,d)\big\||x|^{b}f(x)\big\|_{p,d\mu_{k}},
\]
which gives  the condition $a+b=d_k$.
\end{proof}

Similarly, we prove the sharp Hardy's inequality for Bellman transform.

\begin{theorem}\label{thm3.2}
Let \, $d\in \mathbb{N}$ and $1\le p\le\infty$.
Inequality \eqref{eq3.2} holds with
$\textup{{\textbf{c}}}_k^B(a,b,p,d)<\infty$ if and only if
 $\;\frac{a}{p'}<\frac{b}{p}$ and $\;a+b=d_k$. Moreover,
\[
\textup{{\textbf{c}}}_k^B(a,b,p,d)=\frac{b_{\lambda_k}}{
\frac{b}{p}-\frac{a}{p'}
}.
\]
\end{theorem}

\begin{proof}
We only sketch the proof. 
Considering
\[
\wt{B}f(x)=\int_{|y|\ge |x|}|y|^{d_k/p'-b}f(y)\,dm_k(y)
\]
and \eqref{eq2.1}, we rewrite  inequality \eqref{eq3.2} as follows 
\[
b_{\lambda_k}\||x|^{-a+d_k/p}\wt{B}f(x)\|_{p,dm_{k}}\leq {\textbf{c}}_k^B(a,b,p,d)\||f\|_{p,dm_{k}}.
\]
Then we have 
\[
\wt{B}f(x)=r^{d_k/p'-b}\int_{\mathbb{R}^d}f((r/t)y')g_0(t)\,dm_k(ty'),
\]
where
\[
g_0(t)=t^{b-d_k/p'}\chi_{[0,1]}(t).
\]
Finally,
\[
{\textbf{c}}_k^B(a,b,p,d)= b_{\lambda_k}\|g_0\|_1=b_{\lambda_k}\int_{0}^{1}t^{b-d_k/p'}\,\frac{dt}{t}=\frac{b_{\lambda_k}}{\frac{b}{p}-\frac{a}{p'}}.
\qedhere
\]
\end{proof}

\medskip

\section{Proof of Theorem~\ref{thm1.3}}

{Recall that we consider the case $k\not\equiv 0$,
$\lambda_{k}>-1/2$ and $d_{k}>1$.}
 Let $1<p<\infty$, $\gamma<\frac{d_k}{p}$,
$\beta<\frac{d_k}{p'}$, $\alpha>0$, and $\alpha=\gamma+\beta$. Consider the modified operator
\[
\wt{I}_{\alpha}^kf(x)=\int_{\mathbb{R}^{d}}f(y)|y|^{d_k/p'-\beta}\Phi(x,y)\,dm_k(y).
\]
According to \eqref{eq2.1},  inequality \eqref{eq1.6} for $q=p$ is equivalent to  
\[
b_{\lambda_k}\big\||x|^{-\gamma+d_k/p}\wt{I}_{\alpha}^kf(x)\big\|_{p,dm_{k}}\leq \mathbf{c}_k(\alpha,\beta,\gamma,p,p,d)\|f(x)\|_{p,dm_{k}}.
\]
If $x=rx'$, $y=ty'$, then using the change of variables $y\to (r/t)y'$ and applying  the properties (1), (2) in Lemma~\ref{lem2.3}, we have
\[
\wt{I}_{\alpha}^kf(x)=r^{-\beta+\alpha-d_k/p}\int_{\mathbb{R}^{d}}f((r/t)y')\Phi_1(t,x',y')\,dm_k(ty'),
\]
where
\[
\Phi_1(t,x',y')=t^{d_k/p-\alpha+\beta}\Phi(tx',y').
\]
Hence, by \eqref{eq2.2},
\begin{align*}
|x|^{-\gamma+d_k/p}\wt{I}_{\alpha}^kf(x)
=\int_{\mathbb{R}^{d}}f(ty')\Phi_1(r/t,x',y')\,dm_k(ty').
\end{align*}

We set 
\begin{align*}
J&:=\int_{\mathbb{R}^{d}}\int_{\mathbb{R}^{d}}h(rx')f(ty')\Phi_{1}(r/t,x',y')\,dm_{k}(x)\,dm_{k}(y)\\
&=\int_{\mathbb{S}^{d-1}}\int_{\mathbb{S}^{d-1}}\int_{0}^{\infty}\int_{0}^{\infty}
h(rx')f(ty')\Phi_{1}(r/t,x',y')\,d\nu(t)\,d\nu(r)\,d\sigma_{k}(x')\,d\sigma_{k}(y').
\end{align*}
In light of Lemma~\ref{lem2.1} and H\"{o}lder's inequality, we have
\begin{align*}
|J|&\le \int_{\mathbb{S}^{d-1}}\int_{\mathbb{S}^{d-1}}
\Bigl(\int_{0}^{\infty}|h(rx')|^{p'}\,d\nu(r)\Bigr)^{1/p'}
\Bigl(\int_{0}^{\infty}|f(ty')|^{p}\,d\nu(t)\Bigr)^{1/p}\\
&\quad  \int_{0}^{\infty}\Phi_{1}(t,x',y')\,d\nu(t)\,d\sigma_{k}(x')\,d\sigma_{k}(y')\\
&=\int_{\mathbb{S}^{d-1}}\int_{\mathbb{S}^{d-1}}
\Bigl(\int_{0}^{\infty}|h(rx')|^{p'}\,d\nu(r)
\int_{0}^{\infty}\Phi_{1}(t,x',y')\,d\nu(t)\Bigr)^{1/p'}\\
&\quad \Bigl(\int_{0}^{\infty}|f(ty')|^{p}\,d\nu(t)
\int_{0}^{\infty}\Phi_{1}(t,x',y')\,d\nu(t)\Bigr)^{1/p}\,
d\sigma_{k}(x')\,d\sigma_{k}(y')\\
&\le \Bigl(\int_{\mathbb{S}^{d-1}}\int_{\mathbb{S}^{d-1}}
\int_{0}^{\infty}|h(rx')|^{p'}\,d\nu(r)
\int_{0}^{\infty}\Phi_{1}(t,x',y')\,d\nu(t)\,
d\sigma_{k}(x')\,d\sigma_{k}(y')\Bigr)^{1/p'}\\
&\quad \Bigl(\int_{\mathbb{S}^{d-1}}\int_{\mathbb{S}^{d-1}}
\int_{0}^{\infty}|f(ty')|^{p}\,d\nu(t)
\int_{0}^{\infty}\Phi_{1}(t,x',y')\,d\nu(t)\,
d\sigma_{k}(x')\,d\sigma_{k}(y')\Bigr)^{1/p}.
\end{align*}
Taking into account  the properties (1) and (3) of Lemma~\ref{lem2.3}, we have
\[
\int_{\mathbb{S}^{d-1}}\Phi_{1}(t,x',y')\,d\sigma_{k}(x')=
t^{d_k/p-\alpha+\beta}
\int_{\mathbb{S}^{d-1}}\Phi(tx',y')\,d\sigma_{k}(x')
=t^{d_k/p-\alpha+\beta}\Phi_{0}(t,1)
\]
and
\[
\int_{\mathbb{S}^{d-1}}\Phi_{1}(t,x',y')\,d\sigma_{k}(y')=
t^{d_k/p-\alpha+\beta}\Phi_{0}(1,t)=
t^{d_k/p-\alpha+\beta}\Phi_{0}(t,1).
\]
Then, changing the order of integration implies
\begin{align*}
|J|&\le \Bigl(\int_{0}^{\infty}
t^{d_k/p-\alpha+\beta}\Phi_{0}(t,1)\,d\nu(t)\Bigr)^{1/p'}\Bigl(\int_{\mathbb{S}^{d-1}}\int_{0}^{\infty}
|h(rx')|^{p'}\,d\nu(r)\,d\sigma_{k}(x')\Bigr)^{1/p'}\\
&\quad \Bigl(\int_{0}^{\infty}
t^{d_k/p-\alpha+\beta}\Phi_{0}(t,1)\,d\nu(t)\Bigr)^{1/p}\Bigl(\int_{\mathbb{S}^{d-1}}\int_{0}^{\infty}
|f(ty')|^{p}\,d\nu(t)\,d\sigma_{k}(y')\Bigr)^{1/p}\\
&=\int_{0}^{\infty}
t^{d_k/p-\alpha+\beta}\Phi_{0}(t,1)\,d\nu(t)\,
\|h\|_{p',dm_{k}}\|f\|_{p,dm_{k}}.
\end{align*}
Thus,
$$\mathbf{c}_k(\alpha,\beta,\gamma,p,p,d)\le b_{\lambda_k}\int_{0}^{\infty}
t^{d_k/p-\alpha+\beta}\Phi_{0}(t,1)\,d\nu(t).
$$
Since for  radial functions   $f(x)=f_0(|x|)=f_0(r)$ we have  that
\begin{align*}
|x|^{-\gamma+d_k/p}\wt{I}_{\alpha}^kf(x)=\int_{0}^\infty f_0(r/t)
t^{d_k/p-\alpha+\beta}\Phi_0(t,1)\,\frac{dt}t,
\end{align*}
Lemma~\ref{lem2.1} gives
$$\mathbf{c}_k(\alpha,\beta,\gamma,p,p,d)= b_{\lambda_k}\int_{0}^{\infty}
t^{d_k/p-\alpha+\beta}\Phi_{0}(t,1)\,d\nu(t).
$$

 Let us now prove that the conditions
$\gamma<\frac{d_k}{p}$,
$\beta<\frac{d_k}{p'}$, and $\gamma+\beta=\alpha>0$
 guarantee
 that  $\mathbf{c}_k(\alpha,\beta,\gamma,p,p,d)<\infty$.
We have
\begin{align*}
&\mathbf{c}_k(\alpha,\beta,\gamma,p,p,d)\\ &=
(\gamma_{\alpha}^{k})^{-1}c_{\lambda_{k}}b_{\lambda_{k}}\int_{0}^{\infty}
t^{d_k/p-\alpha+\beta}
\int_{0}^{\pi}\bigl(t^{2}+1-2t\cos \varphi\bigr)^{(\alpha-d_k)/2}
\sin^{d_k-2}\varphi\,d\varphi\,d\nu(t)\\
&=(\gamma_{\alpha}^{k})^{-1}c_{\lambda_{k}}b_{\lambda_{k}}\int_{0}^{\infty}
\frac{t^{d_k/p-\alpha+\beta}}{(1+t^{2})^{(d_k-\alpha)/2}}
\int_{0}^{\pi}\Bigl(1-\frac{2t\cos \varphi}{1+t^{2}}\Bigr)^{(\alpha-d_k)/2}
\sin^{d_k-2}\varphi\,d\varphi\,d\nu(t).
\end{align*}
The  integral with respect to $t$ has singularities at  $t=0,1,\infty$.
It converges
at the origin if and only if $\gamma=\alpha-\beta<\frac{d_k}{p}$. Moreover, the integral converges
at $\infty$ if and only if $\beta<\frac{d_k}{p'}$.
{Concerning the point $t=1$, we set $r:=2t/(1+t^{2})$ and note that, letting $r\to 1-0$,}
\begin{align*}
\psi(r)&:=\int_{0}^{\pi}(1-r\cos \varphi)^{(\alpha-d_k)/2}
\sin^{d_k-2}\varphi\,d\varphi\\
&\asymp \int_{0}^{1}(1-r+ r\varphi^2/2)^{(\alpha-d_k)/2}
\varphi^{d_k-2}\,d\varphi+1\\
&\asymp  \int_{0}^{\sqrt{1-r}}(1-r)^{(\alpha-d_k)/2}
\varphi^{d_k-2}\,d\varphi+ \int_{\sqrt{1-r}}^{1}\varphi^{\alpha-2}
\,d\varphi+1\\
&\asymp \begin{cases}
(1-r)^{\frac{\alpha-1}{2}},& 0<\alpha<1,\\
-\ln{(1-r)},& \alpha=1,\\
\hfill 1,\hfill & \alpha>1.
\end{cases}
\end{align*}
Therefore, letting $t\to 1$, we have
\[
\int_{0}^{\pi}\Bigl(1-\frac{2t\cos \varphi}{1+t^{2}}\Bigr)^{(\alpha-d_k)/2}
\sin^{d_k-2}\varphi\,d\varphi \asymp
\begin{cases}
|1-t|^{\alpha-1},& 0<\alpha<1,\\
-\ln{|1-t|},& \alpha=1,\\
\hfill 1,\hfill & \alpha>1,
\end{cases}
\]
which implies that the singularity at the point $t=1$ is integrable.

It remains to calculate the integral
$\displaystyle \int_{0}^{\infty}
t^{d_k/p-\alpha+\beta}\Phi_{0}(t,1)\,d\nu(t)$.
Let $t\neq 1$, $r=2t/(1+t^2)$. The series
\begin{align*}
(1-r\cos \varphi)^{(\alpha-d_k)/2}&=\Gamma\Bigl(\frac{\alpha-d_k}{2}+1\Bigr)
\sum_{n=0}^{\infty}\frac{(-1)^n}{\Gamma(n+1)\Gamma(\frac{\alpha-d_k}{2}+1-n)}r^n\cos^n\varphi\\
&=\frac{1}{\Gamma(\frac{d_k-\alpha}{2})}\sum_{n=0}^{\infty}\frac{(-1)^n\Gamma(\frac{d_k-\alpha}{2}+n)}{\Gamma(n+1)}r^n\cos^n\varphi
\end{align*}
converges uniformly on $[0,\pi]$ and
\begin{align*}
\psi(r)&=\int_{0}^{\pi}(1-r\cos \varphi)^{(\alpha-d_k)/2}\sin^{d_k-2}\varphi\,d\varphi\\
&=\frac{1}{\Gamma(\frac{d_k-\alpha}{2})}
\sum_{m=0}^{\infty}\frac{\Gamma(\frac{d_k-\alpha}{2}+2m)}{\Gamma(2m+1)}r^{2m}\int_{0}^{\pi}\cos^{2m}\varphi\sin^{d_k-2}\varphi\,d\varphi\\
&=\frac{1}{\Gamma(\frac{d_k-\alpha}{2})}
\sum_{m=0}^{\infty}\frac{\Gamma(m+\frac{1}{2})\Gamma(\frac{d_k-\alpha}{2}+2m)
\Gamma(\frac{d_k-1}{2})}{\Gamma(2m+1)\Gamma(\frac{d_k}{2}+m)}r^{2m}.
\end{align*}
Since a positive series can be integrated term-by-term, it follows that
\[
\mathbf{c}_k(\alpha,\beta,\gamma,p,p,d)=b_{\lambda_{k}}\int_{0}^{\infty}
t^{d_k/p-\alpha+\beta}\Phi_{0}(t,1)\,d\nu(t)
\]
\[
=\frac{(\gamma_{\alpha}^{k})^{-1}c_{\lambda_{k}}b_{\lambda_{k}}}{\Gamma(\frac{d_k-\alpha}{2})}
\sum_{m=0}^{\infty}\frac{2^{2m}\Gamma(m+\frac{1}{2})\Gamma(\frac{d_k-\alpha}{2}+2m)
\Gamma(\frac{d_k-1}{2})}{\Gamma(2m+1)\Gamma(\frac{d_k}{2}+m)}\int_{0}^{\infty}
\frac{t^{d_k/p-\alpha+\beta+2m-1}}{(1+t^{2})^{(d_k-\alpha)/2+2m}}\,dt.
\]
Taking into account that
\[
\int_{0}^{\infty}\frac{t^{d_k/p-\alpha+\beta+2m-1}}{(1+t^{2})^{(d_k-\alpha)/2+2m}}\,dt=
\frac{\Gamma(\frac{d_k}{2p}+\frac{\beta-\alpha}{2}+m)\Gamma(\frac{d_k}{2p'}-\frac{\beta}{2}+m)
}{2\Gamma(\frac{d_k-\alpha}{2}+2m)}
\]
and
\[
\gamma^k_{\alpha}=\frac{2^{\alpha-d_k/2}\Gamma(\frac{\alpha}{2})}{\Gamma(\frac{d_k-\alpha}{2})},
\quad c_{\lambda_k}=\frac{\Gamma(\frac{d_k}{2})}{\Gamma(1/2)\Gamma(\frac{d_k-1}{2})},\quad
b_{\lambda_k}=\frac{1}{2^{d_k/2-1}\Gamma(\frac{d_k}{2})},
\]
 we arrive at
\[
\mathbf{c}_k(\alpha,\beta,\gamma,p,p,d)=\frac{2^{-\alpha}}{\Gamma(\alpha/2)}\sum_{m=0}^{\infty}
\frac{\Gamma(\frac{d_k}{2p}+\frac{\beta-\alpha}{2}+m)\Gamma(\frac{d_k}{2p'}-\frac{\beta}{2}+m)
}{\Gamma(m+1)\Gamma(\frac{d_k}{2}+m)}.
\]
Letting
\[
a=\frac{d_k}{2p}+\frac{\beta-\alpha}{2},\quad
b=\frac{d_k}{2p'}-\frac{\beta}{2},\quad
c=\frac{d_k}{2},
\]
we write
\[
\mathbf{c}_k(\alpha,\beta,\gamma,p,p,d)=\frac{2^{-\alpha}}{\Gamma(\alpha/2)}\sum_{m=0}^{\infty}
\frac{\Gamma(a+m)\Gamma(b+m)}{\Gamma(1+m)\Gamma(c+m)}.
\]

Using now the hypergeometric  function \cite[Ch.~II]{BeEr53}
\[
F(a,b;c;z)=\sum_{m=0}^{\infty}\frac{(a)_{m}(b)_{m}}{m!(c)_{m}}\,z^{k},\quad
(a)_{m}=\frac{\Gamma(a+m)}{\Gamma(a)},
\]
we obtain that
\[
\mathbf{c}_k(\alpha,\beta,\gamma,p,p,d)=\frac{2^{-\alpha}}{\Gamma(\alpha/2)}\,\frac{\Gamma(a)\Gamma(b)}{\Gamma(c)}\,
F(a,b;c;1).
\]
Finally, since \cite[Sect.~2.8, (46)]{BeEr53}
\[
F(a,b;c;1)=\frac{\Gamma(c)\Gamma(c-a-b)}{\Gamma(c-a)\Gamma(c-b)},\quad
c\ne 0,-1,-2,\dots,\quad c>a+b,
\]
we have
\[
\mathbf{c}_k(\alpha,\beta,\gamma,p,p,d)=\frac{2^{-\alpha}\Gamma(a)\Gamma(b)\Gamma(c-a-b)}{\Gamma(\alpha/2)\Gamma(c-a)\Gamma(c-b)},
\]
where
\begin{align*}
c-a-b&=\frac{d_k}{2}-\Bigl(\frac{d_k}{2p}+\frac{\beta-\alpha}{2}+
\frac{d_k}{2p'}-\frac{\beta}{2}\Bigr)=\frac{\alpha}{2},
\\
c-a&=\frac{d_k}{2}-\Bigl(\frac{d_k}{2p}+\frac{\beta-\alpha}{2}\Bigr)=
\frac{d_k}{2p'}+\frac{\alpha-\beta}{2},
\\
c-b&=\frac{d_k}{2}-\Bigl(\frac{d_k}{2p'}-\frac{\beta}{2}\Bigr)=
\frac{d_k}{2p}+\frac{\beta}{2},
\end{align*}
or, equivalently,
\[
\mathbf{c}_k(\alpha,\beta,\gamma,p,p,d)=
2^{-\alpha}\,\frac{\Gamma(\frac{1}{2}(\frac{d_k}{p}-\gamma))
\Gamma(\frac{1}{2}(\frac{d_k}{p'}-\beta))}
{\Gamma(\frac{1}{2}(\frac{d_k}{p'}+\gamma))
\Gamma(\frac{1}{2}(\frac{d_k}{p}+\beta))}.
\reqnomode \tag*{$\Box$}
\]

\begin{remark}
It is clear that the condition $\alpha=\gamma+\beta$ is necessary for
$\mathbf{c}_k(\alpha,\beta,\gamma,p,p,d)<\infty$
to hold.
 Indeed, setting $f_t(x)=f(tx)$, we have
\[
\mathcal{F}_k(f_t)(z)=t^{-d_k}\mathcal{F}_k(f)\Bigl(\frac{z}{t}\Bigr),\quad \tau^yf_t(x)=\tau^{ty}f(tx),\quad I_{\alpha}^kf_t(x)=t^{-\alpha}(I_{\alpha}^kf)_t(x),
\]
\[
\||x|^{\beta}f_t(x)\|_{p,d\mu_{k}}=t^{-\beta-d_k/p}\||x|^{\beta}f(x)\|_{p,d\mu_{k}}.
\]
Writing  inequality \eqref{eq1.6} with $q=p$ as follows
\[
t^{\gamma-\alpha-d_k/p}\||x|^{-\gamma}I_{\alpha}^kf(x)\|_{p,d\mu_{k}}\leq t^{-\beta-d_k/p}\mathbf{c}_k(\alpha,\beta,\gamma,p,p,d)\||x|^{\beta}f(x)\|_{p,d\mu_{k}}
\]
implies  $\alpha=\gamma+\beta.$
\end{remark}

\section{Proof of Theorem~\ref{thm1.4}}

\textbf{Part (a).}
Let $1<p<q<\infty$,
$\gamma<\frac{d_k}{q}$, $\beta<\frac{d_k}{p'}$, $\gamma+\beta\geq 0$, $0<\alpha<d_k$, and $\alpha-\gamma-\beta=d_k(\frac{1}{p}-\frac{1}{q})$. Note that the case $q=p$ was studied  in Theorem~\ref{thm1.3}.
We will use the representation of the kernel $\Phi(x,y)$ given in Lemma~\ref{lem2.3} and then
essentially follow the ideas of
 \cite{SteWei58}.

We write
\[
\wt{I}_{\alpha}^kf(x)=\int_{\mathbb{R}^{d}}f(y)|y|^{-\beta}\Phi_{\alpha}(x,y)\,d\mu_k(y),
\]
where
\[
f\in\mathcal{S}(\mathbb{R}^{d}),\quad \Phi_{\alpha}(x,y)=\int_{\mathbb{R}^{d}}(|x|^{2}+|y|^{2}-2\<y,\eta\>)^{(\alpha-d_k)/2}
\,d\mu_{x}^{k}(\eta),
\]
and
\[
\supp \mu_x^k\subset B_{|x|}=\{\eta\colon |\eta|\le |x|\}.
\]

We define
\[
J:=\int_{\mathbb{R}^{d}}\int_{\mathbb{R}^{d}}f(y)g(x)\,
\frac{\Phi_{\alpha}(x,y)}{|x|^{\gamma}|y|^{\beta}}\,d\mu_{k}(y)\,d\mu_{k}(x).
\]
 It is sufficient to prove the inequality
\begin{equation}\label{eq4.1}
J\lesssim\|f\|_{p,d\mu_{k}}\|g\|_{q',d\mu_{k}}
\end{equation}
for $f,g\ge 0$.

Recall that in the case $1<p<q<\infty$, $\gamma=\beta=0$, and $\alpha=d_k(\frac{1}{p}-\frac{1}{q})$ inequality \eqref{eq4.1} holds  (see \cite{AbdLit15,GorIvaTik17}).
Let
\[
\mathbb{R}^{d}\times \mathbb{R}^{d}=E_{1}\sqcup E_{2}\sqcup E_{3},
\]
where
\begin{align*}
E_{1}&=\{(x,y)\colon 2^{-1}|y|<|x|<2|y|\},\\
E_{2}&=\{(x,y)\colon |x|\le 2^{-1}|y|\},\\
E_{3}&=\{(x,y)\colon |y|\le 2^{-1}|x|\}.
\end{align*}
Then
\[
J=\iint_{E_{1}}+\iint_{E_{2}}+\iint_{E_{3}}=J_{1}+J_{2}+J_{3}.
\]

\smallbreak

\underline{Estimate of $J_1$.}\quad If $(x,y)\in E_{1}$, using $|\eta|\le |x|$,  then by conditions $\alpha-\beta-\gamma=d_k(\frac{1}{p}-\frac{1}{q})$, $\gamma+\beta\ge 0$ we have
\begin{align*}
(|x|^{2}+|y|^{2}-2\<y,\eta\>)^{\frac{\gamma+\beta}{2}}&\le
(|x|^{2}+4|x|^{2}+2|x||y|)^{\frac{\gamma+\beta}{2}}\\&\lesssim
|x|^{\gamma+\beta}\lesssim|x|^{\gamma}|y|^{\beta}
\end{align*}
and
\begin{align*}
\frac{(|x|^{2}+|y|^{2}-2\<y,\eta\>)^{\frac{\alpha-d_k}{2}}}
{|x|^{\gamma}|y|^{\beta}}&\lesssim
(|x|^{2}+|y|^{2}-2\<y,\eta\>)^{\frac{\alpha-\beta-\gamma-d_k}{2}}
\\
&=(|x|^{2}+|y|^{2}-2\<y,\eta\>)^{(d_k(\frac{1}{p}-\frac{1}{q})-d_k)/2}.
\end{align*}
Set $\wt{\alpha}=d_k(\frac{1}{p}-\frac{1}{q})$.
By \eqref{eq4.1} with $\gamma=\beta=0$ and
$0<\wt{\alpha}<d_k$, we have
\[
J_{1}\lesssim \int_{\mathbb{R}^{d}}\int_{\mathbb{R}^{d}}f(y)g(x)\,
\Phi_{\wt{\alpha}}(x,y)\,d\mu_{k}(y)\,d\mu_{k}(x)\lesssim
\|f\|_{p,d\mu_{k}}\|g\|_{q',d\mu_{k}}.
\]

\smallbreak
\underline{Estimate of $J_2$.}\quad If $(x,y)\in E_{2}$, then
\[
\sqrt{|x|^{2}+|y|^{2}-2\<y,\eta\>}\ge
\sqrt{|x|^{2}+|y|^{2}-2|x||y|}\ge |y|-|x|\ge 2^{-1}|y|,
\]
therefore
\begin{align*}
\Phi_{\alpha}(x,y)&=\int_{\mathbb{R}^{d}}
\frac{1}{(\sqrt{|x|^{2}+|y|^{2}-2\<y,\eta\>})^{d_k-\alpha}}\,d\mu_{x}^{k}(\eta)\\
&\lesssim |y|^{\alpha-d_k}
\int_{\mathbb{R}^{d}}d\mu_{x}^{k}(\eta)=|y|^{\alpha-d_k}.
\end{align*}
From here and since $E_{2}\subset \{(x,y)\colon |x|\le |y|\}$,
\begin{align*}
J_{2}&\lesssim \iint_{|x|\le |y|}\frac{f(y)g(x)}
{|x|^{\gamma}|y|^{\beta-\alpha+d_k}}\,d\mu_{k}(x)\,d\mu_{k}(y)\\
&=\int_{\mathbb{R}^{d}}f(y)|y|^{\alpha-\beta-d_k}
\int_{|x|\le |y|}g(x)|x|^{-\gamma}\,d\mu_{k}(x)\,d\mu_{k}(y)\\
&=\int_{\mathbb{R}^{d}}f(y)|y|^{\alpha-\beta-\gamma}Vg(y)\,d\mu_{k}(y),
\end{align*}
where
\[
Vg(y)=|y|^{\gamma-d_k}\int_{|x|\le |y|}
g(x)|x|^{-\gamma}\,d\mu_{k}(x).
\]
Note that
\begin{align*}
Vg(y)&\le |y|^{\gamma-d_k}
\Bigl(\int_{|x|\le |y|}|x|^{-q\gamma}\,d\mu_{k}(x)\Bigr)^{1/q}
\|g\|_{q',d\mu_{k}}\\
&\lesssim |y|^{\gamma-d_k}|y|^{d_k/q-\gamma}
\|g\|_{q',d\mu_{k}}=|y|^{-d_k/q'}\|g\|_{q',d\mu_{k}}.
\end{align*}
Hence
\[
|Vg(y)|^{p'-q'}|y|^{(\alpha-\beta-\gamma)p'}\lesssim
|x|^{-d_k(p'-q')/q'+(\alpha-\beta-\gamma)p'}
\|g\|_{q',d\mu_{k}}^{p'-q'}.
\]
Since
\begin{align*}
-\frac{d_k(p'-q')}{q'}+(\alpha-\beta-\gamma)p'
&=p'\Bigl\{\alpha-\beta-\gamma-d_k\Bigl(\frac{1}{q'}-\frac{1}{p'}\Bigr)\Bigr\}\\
&=p'\Bigl\{\alpha-\beta-\gamma+d_k\Bigl(\frac{1}{q}-\frac{1}{p}\Bigr)\Bigr\}=0,
\end{align*}
it follows that
\begin{equation}\label{eq4.2-}
|Vg(y)|^{p'-q'}|y|^{(\alpha-\beta-\gamma)p'}\lesssim
\|g\|_{q',d\mu_{k}}^{p'-q'}.
\end{equation}
On the other hand, by Theorem~\ref{thm3.1} with $a=d_k-\gamma$, $b=\gamma$, $p=q'$, and $\frac{a}{q}>\frac{b}{q'}$ (or, equivalently,
 $\gamma<\frac{d_k}{q}$), we see that
\begin{equation}\label{eq4.2}
\|Vg\|_{q',d\mu_{k}}\lesssim \|g\|_{q',d\mu_{k}}.
\end{equation}
Using \eqref{eq4.2-} and \eqref{eq4.2},
we have
\begin{align*}
\int_{\mathbb{R}^{d}}|Vg(y)|^{p'}|y|^{(\alpha-\beta-\gamma)p'}\,d\mu_{k}(y)&=
\int_{\mathbb{R}^{d}}|Vg(y)|^{q'}|Vg(y)|^{p'-q'}|y|^{(\alpha-\beta-\gamma)p'}\,d\mu_{k}(y)\\
&\lesssim \int_{\mathbb{R}^{d}}|Vg(y)|^{q'}\,d\mu_{k}(y)\,\|g\|_{q',d\mu_{k}}^{p'-q'}\lesssim
\|g\|_{q',d\mu_{k}}^{p'}.
\end{align*}
This gives
$$
J_{2}
\lesssim \|f\|_{p,d\mu_{k}}
\bigl\||y|^{\alpha-\beta-\gamma}Vg(y)\bigr\|_{p',d\mu_{k}}\lesssim \|f\|_{p,d\mu_{k}}\|g\|_{q',d\mu_{k}}.
$$

Note  that, for $p=1$, a similar result is valid as well, i.e.,
\begin{equation}\label{eq4.3}
J_{2}\lesssim \|f\|_{1,d\mu_{k}}\|g\|_{q',d\mu_{k}},\qquad \gamma<\frac{d_k}{q}, \,\beta\le 0,
\end{equation}
 since $\alpha-\beta-\gamma=d_k/q'$ and
\[
|y|^{\alpha-\beta-\gamma}Vg(y)\lesssim |y|^{\alpha-\beta-\gamma}|y|^{-d_k/q'}\|g\|_{q',d\mu_{k}}=\|g\|_{q',d\mu_{k}}.
\]

\smallbreak
\underline{Estimate of $J_3$.}\quad If $(x,y)\in E_{3}$, we similarly have
\[
\sqrt{|x|^{2}+|y|^{2}-2\<y,\eta\>}\ge 2^{-1}|x|,
\]
\[
\Phi_{\alpha}(x,y)=\int_{\mathbb{R}^{d}}
\frac{1}{(\sqrt{|x|^{2}+|y|^{2}-2\<y,\eta\>})^{d_k-\alpha}}\,d\mu_{x}^{k}(\eta)\lesssim |x|^{\alpha-d_k}
\]
and
\[
J_{3}\lesssim \iint_{|y|\le |x|}\frac{f(y)g(x)}
{|x|^{\gamma-\alpha+d_k}|y|^{\beta}}\,d\mu_{k}(x)\,d\mu_{k}(y)=\int_{\mathbb{R}^{d}}g(x)|x|^{\alpha-\beta-\gamma}Vf(x)\,d\mu_{k}(x),
\]
where
\[
Vf(x)=|x|^{\beta-d_k}\int_{|y|\le |x|}
f(y)|y|^{-\beta}\,d\mu_{k}(y).
\]
Since
\begin{align*}
Vf(x)&\le |x|^{\beta-d_k}
\Bigl(\int_{|y|\le |x|}|y|^{-p'\beta}\,d\mu_{k}(y)\Bigr)^{1/p'}
\|f\|_{p,d\mu_{k}}\\
&\lesssim |x|^{\beta-d_k}|x|^{d_k/p'-\beta}
\|f\|_{p,d\mu_{k}}=|x|^{-d_k/p}\|f\|_{p,d\mu_{k}},
\end{align*}
we obtain
\[
|Vf(x)|^{q-p}|x|^{(\alpha-\beta-\gamma)q}\lesssim
|x|^{-d_k(q-p)/p+(\alpha-\beta-\gamma)q}
\|f\|_{p,d\mu_{k}}^{q-p}=\|f\|_{p,d\mu_{k}}^{q-p}.
\]
Taking into account Theorem~\ref{thm3.1} with $a=d_k-\beta$, $b=\beta$, $p=p$,  $\frac{a}{p'}>\frac{b}{p}$ (or,  $\beta<\frac{d_k}{p'}$), we obtain
$$
\|Vf\|_{p,d\mu_{k}}\lesssim \|f\|_{p,d\mu_{k}},
$$
which implies 
\begin{align*}
\int_{\mathbb{R}^{d}}|Vf(x)|^{q}|x|^{(\alpha-\beta-\gamma)q}\,d\mu_{k}(x)&=
\int_{\mathbb{R}^{d}}|Vf(x)|^{p}|Vf(x)|^{q-p}|x|^{(\alpha-\beta-\gamma)q}\,d\mu_{k}(x)\\
&\lesssim \int_{\mathbb{R}^{d}}|Vf(x)|^{p}\,d\mu_{k}(x)\,\|f\|_{p,d\mu_{k}}^{q-p}\lesssim
\|f\|_{p,d\mu_{k}}^{q}.
\end{align*}

We finally  have
\[
J_{3}\lesssim \|g\|_{q',d\mu_{k}}
\bigl\||x|^{\alpha-\beta-\gamma}Vf(x)\bigr\|_{q,d\mu_{k}}\lesssim
\|f\|_{p,d\mu_{k}}\|g\|_{q',d\mu_{k}}.
\]
Again, the estimate $J_{3}\lesssim \|f\|_{p,d\mu_{k}}\|g\|_{q',d\mu_{k}}$ also holds for $p=1$, $\gamma<\frac{d_k}{q}$, and $\beta<
0$.

This completes the proof of  part (a).

\smallbreak
\textbf{Part (b).} Let us prove the weak $(L^q,L^1)$-boundedness of D-Riesz potential for $1<q<\infty$, $\gamma<\frac{d_k}{q}$, $\beta<0$, $\gamma+\beta\ge 0$, $0<\alpha<d_k$.
We will use the notations and assumptions of  part (a).

It is sufficient to prove the inequality
\begin{equation}\label{eq4.6}
S=\int_{\{x\in \mathbb{R}^d\colon |x|^{-\gamma}\wt{I}_{\alpha}^kf(x)>\lambda\}}\,d\mu_k(x)\lesssim \Bigl(\frac{\|f\|_{1,d\mu_k}}{\lambda}\Bigr)^q.
\end{equation}

Let us consider the operators
\[
A_{\alpha}^if(x)=\int_{\mathbb{R}^{d}}f(y)|y|^{-\beta}\Phi_{\alpha}(x,y)\chi_{E_i}(x,y)\,d\mu_k(y),\quad i=1,2,3.
\]
We have
\[
\wt{I}_{\alpha}^k=\sum_{i=1}^3 A_{\alpha}^i,
\]
and
\[
S=\sum_{i=1}^3\int_{\{x\in \mathbb{R}^d\colon |x|^{-\gamma}A_{\alpha}^if(x)>\lambda/3\}}\,d\mu_k(x)=S_1+S_2+S_3.
\]

\underline{Estimate of $S_1$.} Applying the estimate of $J_1$ and  inequality \eqref{eq4.6} with $1<q<\infty$, $\gamma=\beta=0$, and $\alpha=\wt{\alpha}=\frac{d_k}{q'}$ (see \cite{AbdLit15,GorIvaTik17}), we derive
\[
A_{\alpha}^1f(x)\lesssim\int_{\mathbb{R}^{d}}f(y)|y|^{-\beta}\Phi_{\wt{\alpha}}(x,y)\,d\mu_k(y)=\wt{I}_{\wt{\alpha}}^kf(x),
\]
and
\begin{align*}
S_1&=\int_{\{x\in \mathbb{R}^d\colon |x|^{-\gamma}A_{\alpha}^1f(x)>\lambda/3\}}\,d\mu_k(x)\\&\lesssim \int_{\{x\in \mathbb{R}^d\colon |x|^{-\gamma}\wt{I}_{\wt{\alpha}}^kf(x)\gtrsim\lambda\}}\,d\mu_k(x)\lesssim \Bigl(\frac{\|f\|_{1,d\mu_k}}{\lambda}\Bigr)^q.
\end{align*}

\underline{Estimate of $S_2$.} Applying the obtained  estimate of $J_2$, we get
\[
A_{\alpha}^2f(x)\lesssim\int_{|y|\ge |x|}|y|^{\alpha-\beta-d_k}f(y)\,d\mu_k(y)=B_1f(x).
\]
Since
\[
\int_{\mathbb{R}^d}g(x)|x|^{-\gamma}\int_{|y|\ge |x|}|y|^{\alpha-\beta-d_k}f(y)\,d\mu_k(y)\,d\mu_k(x)
\]
\[
=\int_{\mathbb{R}^{d}}f(y)|y|^{\alpha-\beta-d_k}
\int_{|x|\le |y|}g(x)|x|^{-\gamma}\,d\mu_{k}(x)\,d\mu_{k}(y),
\]
in  light of \eqref{eq4.3} with $p=1$, $\gamma<\frac{d_k}{q}$, and $\beta\le 0$, we have
\[
\big\||x|^{-\gamma}B_1f(x)\big\|_{q,d\mu_k}\lesssim \|f\|_{1,d\mu_k}.
\]
Hence,
\begin{align*}
S_2&=\int_{\{x\in \mathbb{R}^d\colon |x|^{-\gamma}A_{\alpha}^2f(x)>\lambda/3\}}\,d\mu_k(x)\\&\lesssim \int_{\{x\in \mathbb{R}^d\colon |x|^{-\gamma}B_1f(x)\gtrsim\lambda\}}\,d\mu_k(x)\lesssim \Bigl(\frac{\|f\|_{1,d\mu_k}}{\lambda}\Bigr)^q.
\end{align*}

\underline{Estimate of $S_3$.} Applying the estimate of $J_3$, we obtain
\[
A_{\alpha}^3f(x)\lesssim |x|^{\alpha-d_k}\int_{|y|\le |x|}|y|^{-\beta}f(y)\,d\mu_k(y)=H_1f(x).
\]
Using the estimate $J_{3}\lesssim \|f\|_{1,d\mu_{k}}\|g\|_{q',d\mu_{k}}$ with $\gamma<\frac{d_k}{q}$ and $\beta<0$ yields
\[
\big\| |x|^{-\gamma}H_1f(x)\big\|_{q,d\mu_k}\lesssim \|f\|_{1,d\mu_k}.
\]
Thus,
\begin{eqnarray}\nonumber
S_3&=&\int_{\{x\in \mathbb{R}^d\colon |x|^{-\gamma}A_{\alpha}^3f(x)>\lambda/3\}}\,d\mu_k(x)
\\
\label{eq4.7}
&\lesssim& \int_{\{x\in \mathbb{R}^d\colon |x|^{-\gamma}H_1f(x)\gtrsim\lambda\}}\,d\mu_k(x)\lesssim \Big(\frac{\|f\|_{1,d\mu_k}}{\lambda}\Bigr)^q.
\end{eqnarray}
If $\beta=0$, $\alpha-\gamma=d_k(1-\frac{1}{q})$, then
\[
|x|^{-\gamma}H_1f(x)=|x|^{\alpha-\gamma-d_k}\int_{|y|\le |x|}f(y)\,d\mu_k(y)=|x|^{-d_k/q}\int_{|y|\le |x|}f(y)\,d\mu_k(y).
\]
Since  inequality \eqref{eq4.7} is homogeneous, we can assume that $\|f\|_{1,d\mu_k}=1$.
Therefore,
\begin{align*}
S_3&\lesssim\int_{\{x\in \mathbb{R}^d\colon |x|^{-d_k/q}\int_{|y|\le |x|}f(y)\,d\mu_k(y)\gtrsim\lambda\}}\,d\mu_k(x)\\
&\lesssim\int_{\{x\in \mathbb{R}^d\colon |x|^{-d_k/q}\gtrsim\lambda\}}\,d\mu_k(x)\lesssim \lambda^{-q}=\Big(\frac{\|f\|_{1,d\mu_k}}{\lambda}\Bigr)^q,
\end{align*}
completing the proof.
\hfill $\Box$

\smallbreak
Let us mention that for the so-called B-Riesz potentials the results that are similar to Theorem~\ref{thm1.4} were established in \cite{GadGul11}.

\section{Properties of the spaces $\Phi_k$ and $\Psi_k$}

Recall that $T_{j}$, $j=1,\ldots,d$, are differential-differences Dunkl operators given by (\ref{eq1.2-}), $T^n=\prod_{j=1}^dT_j^{n_j}$, $n\in\mathbb{Z}^d_+$,
\[
\Phi_k=\Bigl\{f\in \mathcal{S}(\mathbb{R}^d)\colon
\int_{\mathbb{R}^{d}}x^nf(x)\,d\mu_k(x)=0,\quad n\in \mathbb{Z}^d_+\Bigr\},
\]
and
\[
\Psi_k=\big\{\mathcal{F}_{k}(f)\colon f\in\Phi_k\big\}.
\]

If $f\in \mathcal{S}(\mathbb{R}^d)$, then from the definition of the generalized exponential function $e_k(x,y)$ and 
\[
f(x)=\int_{\mathbb{R}^d}e_k(x,y)\mathcal{F}_k(f)(y)\,d\mu_k(y),
\]
we obtain that
\[
T^nf(x)=i^n\int_{\mathbb{R}^d}y^ne_k(x,y)\mathcal{F}_k(f)(y)\,d\mu_k(y),\quad n\in\mathbb{Z}^d_+,
\]
and
\[
T^nf(0)=i^n\int_{\mathbb{R}^d}y^n\mathcal{F}_k(f)(y)\,d\mu_k(y).
\]
Therefore,
\[
\Psi_k=\big\{f\in\mathcal{S}(\mathbb{R}^{d})\colon T^nf(0)=0,\,\, n\in \mathbb{Z}^d_+\big\}.
\]
Note that
in the classical case ($k\equiv0$) we have
\[
\Psi=\Psi_0=\big\{\mathcal{F}(f)\colon f\in\Phi\big\}=\big\{f\in\mathcal{S}(\mathbb{R}^{d})\colon D^nf(0)=0,\,\, n\in \mathbb{Z}^d_+\big\}.
\]

\begin{theorem}\label{thm6.1}
We have $\Psi_k=\Psi$.
\end{theorem}
\begin{proof} Let $f\in \Psi$, $D=(D_1,\dots,D_d)$, and let $\partial_af(x)=\<Df(x),\frac{a}{|a|}\>$ be
the directional derivative with respect to a vector $a$. Taking into consideration
\[
\frac{f(x)-f(\sigma_{a}x)}{\<a,x\>}=\frac{2}{|a|}\int_{0}^{1}\partial_af\Bigl(x-\frac{2t\<a,x\>}{|a|^2}a\Bigr)\,dt
\]
and
\begin{equation}\label{eq5.1}
T_{j}f(x)=D_{j}f(x)+
\sum_{a\in R_{+}}\frac{2k(a)\<a,e_{j}\>}{|a|}\int_{0}^{1}\partial_af\Bigl(x-\frac{2t\<a,x\>}{|a|^2}a\Bigr)\,dt
\end{equation}
we obtain that $T_{j}f(0)=0$, $j=1,\dots,d$.
By \eqref{eq5.1}, we derive that  $\Psi\subset \Psi_k$.
In addition, if $D^nf(0)=0$ for $|n|=\sum_{j=1}^dn_j\le m$, then $T^nf(0)=0$ for $|n|\le m$.

Let $m\in \mathbb{Z}_+$ and $f\in \Psi_k$ be a real function. Using the Taylor formula, we write
\[
f(x)=p(x)+r(x),
\]
where $p(x)$ is a polynomial of degree $\deg p\le m$, and $D^nr(0)=0$ for $|n|\le m$.
Since $T^nf(0)=T^nr(0)=0$ for $|n|\le m$, it follows that $T^np(0)=0$ for $|n|\le m$ and, in particular, $p(0)=0$.
By \cite{Ros02},
\[
0=p(T)p(0)=\int_{\mathbb{R}^d}\bigl(e^{-\Delta_k/2}p(x)\bigr)^2e^{-|x|^2/2}\,d\mu_k(x)
\]
and $e^{-\Delta_k/2}p(x)=0$.
Since $e^{-\Delta_k/2}$ is a bijective operator on the set of all polynomials \cite{Ros02}, we obtain that $p(x)\equiv 0$, and $D^nf(0)=D^np(0)=0$ for $|n|\le m$. Thus, $\Psi_k\subset \Psi$.
\end{proof}

 Theorem~\ref{thm6.1} immediately implies the following

\begin{corollary}\label{cor6.2}
We have $I^{k}_{\alpha}(\Phi_{k})=\Phi_{k}$ and $\mathcal{F}_{k}(I_{\alpha}^{k})(\Psi_{k})=\Psi_{k}$.
\end{corollary}

Let $f\in\mathcal{S}(\mathbb{R}^d)$. Using the positive $L_p$-bounded generalized translation operator
\[
\mathcal{T}^{t}f(x)=\int_{\mathbb{S}^{d-1}}\tau^{ty'}f(x)\,d\sigma_k(y')
\]
\cite{GorIvaTik17}, we can write the D-Riesz potential and the convolution with a radial function $g_0(|y|)$ as follows
\begin{equation}\label{eq5.2}
I_{\alpha}^kf(x)=(\gamma^k_{\alpha})^{-1}\int_{0}^{\infty}\mathcal{T}^{t}f(x)\,t^{\alpha-d_k}\,d\nu_{\lambda_k}(t)
\end{equation}
and
\begin{equation}\label{eq5.3}
\int_{\mathbb{R}^{d}}\tau^{-y}f(x)\,g_0(|y|)\,d\mu_{k}(y)=\int_0^{\infty}\mathcal{T}^tf(x)g_0(t)\,d\nu_{\lambda_k}(t).
\end{equation}

The proof of the  following result is based on  Theorem~\ref{thm1.3}.

\begin{theorem}\label{thm5.2}
If $1<p<\infty$, $-\frac{d_k}{p}<\beta<\frac{d_k}{p'}$, then
$\Phi_k$ is dense in $L^p(\mathbb{R}^d,  |x|^{\beta p}\,d\mu_k)$.
\end{theorem}
\begin{proof}
Let $\eta\in \mathcal{S}(\mathbb{R}^d)$ be such that $\eta(x)=1$ if $|x|\leq 1$, $\eta(x)>0$ if $|x|<2$, and $\eta(x)=0$ if $|x|\geq 2$. We can assume that $f\in\mathcal{S}(\mathbb{R}^d)$.

Set
\[
\psi_0(|y|)=\mathcal{F}_k(\eta)(y),\quad \psi_{0N}(|y|)=\frac{1}{N^{d_k}}\psi_0\Bigl(\frac{|y|}{N}\Bigr)=\mathcal{F}_k(\eta(N\cdot))(y),
\]
\[
\varphi_N(x)=f(x)-\int_{\mathbb{R}^d}\tau^{-y}f(x)\psi_{0N}(|y|)\,d\mu_k(y).
\]
Since (see \cite{GorIvaTik17})
\[
\mathcal{F}_k(\varphi_N)(z)=(1-\eta(Nz))\mathcal{F}_k(f)(z)\in \Psi_k,
\]
it follows that $\varphi_N\in\Phi_k$ and, by \eqref{eq5.3},
\begin{equation}\label{eq5.4}
\bigl\||x|^{\beta}(f(x)-\varphi_N(x))\bigl\|_{p,d\mu_k}=\Bigl\||x|^{\beta}\int_0^{\infty}\mathcal{T}^tf(x)\psi_{0N}(t)\,d\nu_{\lambda_k}(t)\Bigr\|_{p,d\mu_k}.
\end{equation}

For any $\alpha\in(0,d_k)$, we have
\[
|\psi_0(t)|\lesssim t^{\alpha-d_k}\quad \text{and}\quad |\psi_{0N}(t)|\lesssim N^{-\alpha}t^{\alpha-d_k}.
\]
Hence, by positivity of the operator $T^t$ and \eqref{eq5.2},
\begin{align*}
\Bigl|\int_0^{\infty}\mathcal{T}^tf(x)\psi_{0N}(t)\,d\nu_{\lambda_k}(t)\Bigr|&\le\int_0^{\infty}\mathcal{T}^t|f|(x)\,|\psi_{0N}(t)|\,d\nu_{\lambda_k}(t)\\
&\lesssim N^{-\alpha}\int_0^{\infty}\mathcal{T}^t|f(x)|\,t^{\alpha-d_k}\,d\nu_{\lambda_k}(t)=N^{-\alpha}I_{\alpha}^k|f|(x).
\end{align*}
This, \eqref{eq5.4}, and Theorem~\ref{thm1.3} imply
\begin{align*}
\bigl\||x|^{\beta}(f(x)-\varphi_N(x))\bigl\|_{p,d\mu_k}&\lesssim N^{-\alpha}\bigl\||x|^{\beta}I_{\alpha}^k|f|(x)\bigl\|_{p,d\mu_k}\\
&\lesssim N^{-\alpha} \bigl\||x|^{\delta}f(x)\bigl\|_{p,d\mu_k}\lesssim N^{-\alpha},
\end{align*}
where
$\alpha>0$ is chosen so that $\delta=\alpha+\beta<\frac{d_k}{p'}$.
\end{proof}


\begin{thebibliography}{99}

\bibitem{AbdLit15}
C.~Abdelkefi, M.~Rachdi, \textit{Some properties of the Riesz potentials in
Dunkl analysis}, Ricerche Mat. \textbf {64} (2015), no.~4, 195--215.

\bibitem{BeEr53}
G.~Bateman, A.~Erd\'elyi, et al., \textit{Higher Transcendental Functions, I},
McGraw Hill Book Company, New York, 1953.

\bibitem{BatErd53}
H.~Bateman and A.~Erd\'elyi, \textit{Higher Transcendental Functions}, vol.~2,
New York, MacGraw-Hill, 1953.


\bibitem{Bec08}
W.~Beckner, \textit{Pitt's inequality with sharp convolution estimates}, Proc. Amer. Math. Soc. \textbf {136} (2008), no.~5, 1871--1885.

\bibitem{ChrGra95}
M.~Christ, L.~Grafakos, \textit{Best constants for two nonconvolution inequalities}, Proc. Amer. Math. Soc. \textbf {123} (1995), no.~6, 1687--1693.

\bibitem{Dun92}
C.\,F.~Dunkl, \textit{Hankel transforms associated to finite reflections
groups}, Contemp. Math. \textbf{138} (1992), 123--138.

\bibitem{Fro35}
O.~Frostman, \textit{Potentiel d'equilibre et capacite des ensembles avec quelques applications a la theorie des fonctions}, These, Communic. Semin. Math. de l'Univ. de Lund. \textbf{3 } (1935).

\bibitem{FuGraLuZha12}
Z.\,W.~Fu, L.~Grafakos, S.\,Z.~Lu, F.\,Y.~Zhao, \textit{Sharp bounds for $m$-linear Hardy and Hilbert operators}, Houston Journal of Mathematics. \textbf {38} (2012), no.~1, 225--244.

\bibitem{GadGul11}
A.\,D.~Gadjiev, V.\,S.~Guliyev,  A.~Serbetci and E.\,V.~Guliyev,  \textit{The Stein--Weiss type inequalities for the B-Riesz potentials}, J. Math. Ineq. \textbf{5} (2011), no.~1, 87-106.

\bibitem{GorIvaTik17}
D.\,V.~Gorbachev, V.\,I.~Ivanov,  S.\,Yu.~Tikhonov,  \textit{$L^p$-bounded Dunkl-type generalized translation operator and its applications},     arXiv:1703.06830  (2017).

\bibitem{HarLit28}
G.\,H.~Hardy, J.\,E.~Littelwood, \textit{Some properties of fractional integrals, I}, Math. Zeit. \textbf {27} (1928), 565--606.

\bibitem{HasMusSif09}
S.~Hassani, S.~Mustapha, M.~Sifi, \textit{Riesz potentials and fractional
maximal function for the Dunkl transform}, J. Lie Theory. \textbf{19} (2009),
no.~4, 725--734.

\bibitem{Her77}
I.\,W.~Herbst, \textit{Spectral theory of the operator $(p^2+m^2)^{1/2}-Ze^2/r$}, Comm. Math. Phys. \textbf {53} (1977), 285--294.

\bibitem{Lie83}
E.\,H.~Lieb, \textit{Sharp constants in the Hardy-Littlewood-Sobolev and related inequalities}, Ann. of Math.  \textbf {118} (1983), no.~2, 349--374.

\bibitem{Liz63}
P.\,I.~Lizorkin, \textit{Generalized Liouville differentiation and function spaces $L^r_p(E_n)$. Embedding theorems}, Sbornik: Math. \textbf{60} (1963), no.~3, 325–353. (in Russian)

\bibitem{Pla07}
S.~S.~Platonov, \textit{Bessel harmonic analysis and approximation of functions
on the half-line}, Izvestiya: Math. \textbf{71} (2007), no.~5, 1001--1048.

\bibitem{Rie49}
M.~Riesz, \textit{L'integrale de Riemann-Liouville
et le probleme de Cauchy}, Acta Math. \textbf {81} (1949), no.~1, 1--222.

\bibitem{Ros98}
M.~R\"osler, \textit{Generalized Hermite polynomials and the heat equation for
Dunkl operators}, Comm. Math. Phys. \textbf{192}
(1998),~519--542.

\bibitem{Ros99}
M.~R\"osler, \textit{Positivity of Dunkl's intertwinning operator}, Duke Math. J. \textbf{98}
(1999),~445--463.

\bibitem{Ros02}
M.~R\"osler, \textit{Dunkl operators. Theory and applications, in Orthogonal
Polynomials and Special Functions}, Lecture Notes in Math. Springer-Verlag,
1817, pp.~93--135, 2003.

\bibitem{Ros03}
M.~R\"osler, \textit{A positive radial product formula for the Dunkl kernel}, Trans.
Amer. Math. Soc. \textbf{355}
(2003),~2413--2438.

\bibitem{Sam02}
S.\,G.~Samko, \textit{Hypersingular Integrals and Their Applications}, Series
Analytical Methods and Special Functions \textbf {5}, Taylor, Francis,
London--New York (2005).

\bibitem{Sam05}
S.~Samko, \textit{Best constant in the weighted Hardy inequality: the spatial and spherical version}, Fract. Calc. Anal. Appl.  \textbf {8} (2005), 39--52.

\bibitem{Saw}
E. Sawyer, \textit{A two weight weak type inequality for fractional integrals},
Trans. Am. Math. Soc., \textbf {281} (1984), 339--345.

\bibitem{Sob38}
S.~Soboleff, \textit{On a theorem in functional analysis},
  Rec. Math. [Mat. Sbornik] N.S., 4(46):3  (1938), 471--497; Amer. Math. Soc. Transl., no. 2(34) (1963), 39--68.



\bibitem{SteWei58}
E.\,M.~Stein, G.~Weiss, \textit{Fractional integrals on n-dimensional Euclidean space}, J. Math. Mech. \textbf {7} (1958), no.~4, 503--514.

\bibitem{ThaXu05}
S.~Thangavelu, Y.~Xu, \textit{Convolution operator and maximal function for
Dunkl transform}, J. d'Analyse. Math. \textbf{97} (2005), 25--55.

\bibitem{ThaXu07}
S.~Thangavelu, Y.~Xu, \textit{Riesz transform and Riesz potentials for Dunkl
transform}, J. Comput. Appl. Math. \textbf{199} (2007), 181--195.

\bibitem{Tri02}
K.~Trim\`{e}che, \textit{Paley-Wiener Theorems for the Dunkl transform and
Dunkl translation operators}, Integral Transform. Spec. Funct. \textbf{13} (2002), 17--38.

\bibitem{Xu00}
Y.~Xu, \textit{Dunkl operators: Funk--Hecke formula for orthogonal polynomials on
spheres and on balls}, Bull. London Math. Soc. \textbf{32} (2000), 447--457.



\end{thebibliography}
\end{document}